\documentclass[10pt]{article}

\usepackage{graphicx}%
\usepackage{multirow}%
\usepackage{amsmath,amssymb,amsfonts}%
\usepackage{amsthm}%
\usepackage{mathrsfs}%
\usepackage[title]{appendix}%
\usepackage{xcolor}%
\usepackage{textcomp}%
\usepackage{manyfoot}%
\usepackage{booktabs}%
\usepackage{algorithm}%
\usepackage{algorithmicx}%
\usepackage{algpseudocode}%
\usepackage{listings}%
\usepackage{enumitem,relsize}
\usepackage{indentfirst}
\numberwithin{equation}{section} 

\usepackage{geometry}

\newtheorem{theorem}{Theorem}[section]
\newtheorem{proposition}[theorem]{Proposition}%
\newtheorem{remark}[theorem]{Remark}%
\newtheorem{lemma}[theorem]{Lemma}
\newtheorem{corollary}[theorem]{Corollary}%
\numberwithin{equation}{section}

\newcommand{\md}{\mathrm{d}}
\newcommand{\mr}{\mathbf{r}}
\newcommand{\mm}{\mathbf{m}}
\newcommand{\bS}{{\mathbb{S}^1}}
\newcommand{\bfn}{{\mathbf{n}}}
\newcommand{\qeN}{{\Xi}}
\newcommand{\errc}{{\eta}}
\newcommand{\scrG}{{\mathscr{G}}}
\newcommand{\scrP}{{\mathscr{P}}}
\newcommand{\cQdom}{{\mathcal{Q}_{\mathrm{phys}}}}
\newcommand{\cQdomM}{{\mathcal{Q}_{\mathrm{phys}}^M}}

\allowdisplaybreaks[4]

\begin{document}

\title{Tensor gradient flow with quasi-entropy for smectic liquid crystals and discretizations keeping coupled physical constraints}

\author{ Jie Xu\footnote{SKLMS, Institute of Computational Mathematics and Scientific/Engineering Computing (ICMSEC), Academy of Mathematics and Systems Science (AMSS), Chinese Academy of Sciences, Beijing, China. Email: xujie@lsec.cc.ac.cn}, Xiaomei Yao\footnote{School of Mathematics and Physics, North China Electric Power University, Beijing, China. Email: yaoxm16@ncepu.edu.cn} }

\date{}


\maketitle

\begin{abstract}
  A gradient flow for the concentration and a $2\times 2$ tensor is constructed to describe smectic liquid crystals. 
  The free energy consists of the entropy term and interaction term involving squared second order spatial derivatives.
  The entropy term incorporates the concentration in the quasi-entropy originally proposed for the tensor only, which is a strictly convex and lower semicontinuous function imposing coupled constraints between the concentration and the tensor.
  An evolution equation for the boundary normal derivative of the concentration is proposed in addition to the equations for the concentration and the tensor, giving an energy dissipation system. 
  Numerical schemes are designed with emphases on using the entropy term to keep the coupled constraints, and the discretization of the boundary normal derivatives satisfying summation by parts.
  Existence, uniqueness, energy dissipation and error estimates are established.  Numerical results indicate the efficiency and robustness of the scheme. 
  Configurations of defects different from other layer structures are observed.

\textbf{Keywords.} Smectic liquid crystals; Tensor model; Gradient flow; Coupled constraints; Boundary evolution equations; Stability and error analysis
\end{abstract}

\section{Introduction}
The properties of liquid crystals are determined to a great extent by local
anisotropy typically generated by orientational distribution of non-spherical
molecules with rigidity. 
The local anisotropy may either dominate the orientational structure, resulting
in nematic phases~\cite{deG1993,Stark2001,Skarabot2007two,Mulder2011,
	Beller2015Shape,Majumdar2017,Wang2018,M2019Colloidal,Yin2020Construction}, or
couple with concentration modulation, giving rise to smectic and columnar
phases~\cite{deG1972,Sackmann1989,Leslie1991,Netz1997Interfaces,
	Tsori2000Defects,Duque2002Theory,Mai2012,Hu2014,Nemitz2016,Selmi2017}. Furthermore, the local anisotropy leads to more possibilities
of structural imperfections when external forces, such as geometric
confinements, boundary anchorings or preferences, are present. For nematics,
topological defects of various kinds have been studied
extensively~\cite{Schopohl1988,Mkaddem2000,Lavrentovich2003,Lewis2014,Hu2016On,Garlea2016Finite,Canevari2017,Hashemi2017Fractal,Wang2018,Yin2020Construction,Nys2022,Yao2022,Haputhanthrige2024,Han2024,Bahr2025}. Defects in smectics exhibit distinctive
configurations since they may involve interruptions of both layer structures
and the orientational
order~\cite{Kralj1996,Liang2013,Repula2018,Wittmann2021Particle}.
On the theoretical and computational aspect, some cases are
discussed preliminarily while many more remain to be
examined~\cite{Han2015From,Mei2015On,Xia2021Structural,Xia2023Variational,Shi2025A}.

Computational investigations of smectics can be coarsely classified into
microscopic and macroscopic approaches. Microscopic approaches, such as
molecular simulations of small rigid molecules and self-consistent field
simulations of polymer chains with rigid
monomers~\cite{Wang2021Analytical,Monderkamp2021Topology,Wittmann2021Particle,Cai2022Tilt,Monderkamp2022Topological,Wittmann2023Colloidal},
generally require high computational costs as they need to compute the state of
every single molecule. In macroscopic approaches, the system is typically
governed by a free energy of field variables for both concentration and local
anisotropy, whether the focus is equilibrium states or dynamics. For rod-like
molecules, depending on the variables for the local anisotropy, the models can
be categorized into tensor
models~\cite{Mukherjee2001Simple,Brand2001,Biscari2007,Mei2015On,Ball2015Discontinuous,Xia2021Structural,Paget2022,Xia2023Variational,Ball2023A}
and vector
models~\cite{deG1972,ChenLubensky,Poniewierski1991,Linhananta1991,Abukhdeir2008,Abukhdeir2008Defect,Pevnyi2014Modeling,Zappone2023One}.
The free energy necessarily has terms for the coupling between local anisotropy
and concentration variations, basically written down according to assumptions on layer modulation. Another point to be noted is that vector models
possess fewer variables but are known to have limitations when studying
defects. Hence, it would be more appropriate to use tensor models, which
typically describe the local anisotropy by a second-order symmetric traceless
tensor, to examine defects in smectics.

The free energy in tensor models can also be constructed by expansion from the
microscopic theory, which does not require a priori assumptions on structural
configurations. Such a procedure has been applied to the studies of nematic
phases of different
kinds~\cite{Straley_1974, Bisi_2006, Han2015From, Xu2018A, Wang2021Modelling},
as well as smectic phases of rod-like molecules~\cite{Han2015From,Mei2015On}.
One advantage of this class of free energies is that they include
possible couplings up to certain order, so that they might be more suitable for
structural imperfections. However, to obtain a model ready for the
computational studies, especially dynamics, of defects in smectics, many
efforts need to be made on both modeling and numerical methods.

For the tensor free energy derived from microscopic theory, it is known that
the free energy always contains an entropy term. For nematic phases where the
concentration is constant, the entropy term is derived through the maximum
entropy state under the given value of the tensor. The advantage of having this
entropy term is that it constrains the eigenvalues of the tensor within the
physical range. However, it is an implicit function of the tensor involving
integrals on the directions of rods. Such a formulation poses substantial computational challenges,
since the implicit function needs to be solved at every grid point.
On the other hand, the role
of the entropy term is stabilization making the system tend to the isotropic
state and is independent of the interactions between rod-like molecules. In
this light, it is appropriate to seek for an alternative function in simple
form with the essential properties of the entropy term to take its place. To
this end, the quasi-entropy is proposed~\cite{Xu2022Quasi} for general rigid
molecules in the case of constant concentration. It maintains several crucial
mathematical properties of the original entropy term, including strict
convexity, ability to constrain the covariance matrix positive definite,
rotational invariance, and consistency in symmetry reductions. It is also
verified for several representative rigid molecules that the free energy with
the quasi-entropy captures the underlying physics in the spatially homogeneous
systems. Furthermore, for spatially inhomogeneous systems,
numerical methods are constructed for a gradient flow with emphases on
maintaining the physical range and energy dissipation~\cite{Wang2023Q}.
The evolution of defect patterns also proves to be consistent with previous
results.

In this work, we discuss rod-like molecules exhibiting smectic phases (so that
the concentration varies spatially) confined in a bounded 2D region with their directions lying in a plane. 
In this case, the tensor describing local anisotropy is $2\times 2$ symmetric
traceless, whose physical range of eigenvalues now depends on the
concentration.
It is known that confined systems of the nematic phase typically possess multiple local energy minimizers.
Their configurations are significant whether or not they are global minimizers. 
To explore these energy minimizers and their relations, it is necessary to evolve from different initial states according to certain equations. 
To this end, we aim to construct a tensor gradient flow, and discuss numerical methods maintaining its essential properties.
A gradient flow is an energy dissipative system, 
determined not only by the free energy but also by its dissipation operator.
Here, we choose the dissipation operator as simple as possible, just to ensure
that it keeps the mass conservation and the tensor symmetric traceless. In this
way, we could pay more attention dealing with the free energy.
It shall be
noted that this choice does not imply that the dissipation operator is an
ingredient less significant. Actually, the dissipation operator might be highly
coupled if one considers dynamic models from the microscopic
theory~\cite{Yu2010A,Xu2018}, resulting in other problems to be discussed
later separately.

So, we could concentrate on the tensor free energy. Since the concentration may
vary spatially, the free energy shall be a functional of the concentration and
the tensor. For the free energy derived from microscopic theory, the major
difference from the case with constant concentration lies within the entropy
term. Therefore, one of the core targets in this work is to extend the
quasi-entropy to incorporate concentration variations. We shall write down its
form and show the essential mathematical properties. In particular, the
quasi-entropy for the varying concentration has a domain neither open nor
closed. This arises from the definition when the concentration takes zero. It
turns out that the quasi-entropy cannot be continuous here, but we can still
show that it is strictly convex and lower semicontinuous. Meanwhile, the
quasi-entropy satisfies rotational invariance, and acts as a barrier function
imposing coupled constraints on the concentration and the tensor, which is
analogous to the case of constant concentration.

Another problem arises in the boundary conditions. The description of smectic
phases involves concentration modulation that calls for at least squared
second-order spatial derivative terms in the free energy. If we consider
Dirichlet type boundary conditions, both function values and normal derivatives
are necessary. However, there lacks the information to determine the value of
normal derivatives no matter by measurements or theoretical derivations. In
this situation, we would introduce an evolution equation for the normal
derivatives from the free energy. Combined with the gradient flow for the
concentration and the tensor, we show that these equations form an energy
dissipative system. The quasi-entropy and the gradient flow will be discussed
in Sec.~\ref{GF}.

Next, we propose numerical methods in Sec.~\ref{NM}, including both time and
spatial discretizations, for the gradient flow system.
Since the gradient flow is written down mainly for studying various steady states, more emphasis shall be put on keeping essential structures and properties of the gradient flow. 
In particular, due to the pattern of
the smectic phases, the concentration and the tensor may become very close to
the boundary of the domain. Therefore, it requires special care to maintain the
coupled constraints of the concentration and the tensor. The coupled
constraints are handled by utilizing the properties of the quasi-entropy, which
is treated implicitly. We show that the resulting scheme, first-order in time,
is uniquely solvable within the domain of the concentration and the tensor,
nonzero almost everywhere, and is energy dissipative. In the proof, the
convexity and lower semicontinuity of the quasi-entropy play key roles. The
error estimate is then established with the normal derivatives on the
boundaries included. For the spatial discretization, we consider the finite
difference in a rectangular region. Since both high-order derivatives inside
and normal derivatives on the boundaries are present, the scheme needs delicate
choice to be consistent with the continuous case. This can be done by
discretizing the free energy in each rectangular cell, followed by finding out
the corresponding discretized variational derivatives. To accomplish this
goal, ghost points are necessarily included outside the rectangular region. The
spatial truncation error is second-order. We verify that the finite difference
scheme satisfies the crucial summation by parts involving the normal
derivatives on the boundaries. Armed with these equalities, we show that the fully discrete scheme has a unique solution, with the coupled
constraints for the concentration and the tensor satisfied for the grid points
inside. Following that, the energy dissipation and the error estimates are
established.

The careful construction of numerical methods is indeed necessary, as we
realize from the numerical examples in Sec.~\ref{NE}. If we adopt a simple
linear semi-implicit scheme, the numerical solution easily goes out of the
domain, unless a extremely small time step is chosen. On the contrary, the scheme
proposed in the previous section allows time steps several magnitues larger.
Meanwhile, it does
not require many nonlinear iterations, so that in total it has superiority in
efficiency and stability. Then, a few cases on the evolution of defects in
smectic phases are examined. We observe that the local anisotropy affect to a large
extent the disruption of layer structures, exhibiting evident distinctions from
the defects in layer structures without local anisotropy.

Although in this work we discuss the special case of the concentration coupled
with the $2\times 2$ tensor, the quasi-entropy can be written in the same way
for the $3\times 3$ tensor, and further multiple tensors if nonaxisymmetric
rigid molecules are examined. The quasi-entropy for the concentration coupled
with tensors can also be useful in the dynamic tensor models derived from
microscopic theory. For these problems, we shall give a few concluding remarks
in Sec.~\ref{conclusion}.

\section{Tensor Model and Gradient Flow}\label{GF}
\subsection{Order parameters}
The system we study consists of rigid, rod-like molecules with their directions lying in a plane. 
Let \(f(\mr,\mm)\ge 0\) denote the number density of molecules at position
\(\mr\in\mathbb{R}^2\) with orientation \(\mm\in\bS\). The order parameters are
defined as directional moments of \(f\).
A rod-like molecule is assumed to have the head-to-tail symmetry, so that $f(\mr,-\mm)=f(\mr,\mm)$, which means that the odd-order moments of $\mm$ vanish. 
We introduce zeroth and second moments characterizing the concentration and the local anisotropy at certain position $\mr$, 
\begin{equation*}
  c(\mr)=\int_{\bS}f(\mr,\mm)\,\md\mm,\nonumber\quad
  Q(\mr)=\int_{\bS}\left(\mm\otimes\mm-{1\over2}I\right) f(\mr,\mm)\,\md\mm,
\end{equation*}
where $\otimes$ represents the tensor product, $\md\mm$ is the uniform unit measure on $\bS$ (i.e., $\int_{\bS}\md\mm=1$), and the tensor $\mm\otimes\mm$ can be understood as a $2\times 2$ matrix. 
In the above, we subtract from the second moment a half of the $2\times 2$ identity matrix $I$ for the integrand to make $Q(\mr)$ symmetric traceless. 
For conveniences of further discussions, we define the directional density function $\rho(\mr,\mm)=f(\mr,\mm)/c(\mr)$, which satisfies $\int_{\bS} \rho\md\mm=1$ (if $c(\mr)=0$, $\rho(\mr,\mm)$ can be chosen as any normalized density). 

By the definition, when $c$ is zero, $Q$ is also the zero matrix.
When $c$ is positive, we have $Q+cI/2=\int_{\bS}\mm\otimes\mm f\md\mm$ positive definite.
In other words, the physical range of $(c,Q)$ is given by 
\begin{equation*}
  \cQdom =\Big \{ (c,Q)\Big|c=0,Q=0,\text{ or } c>0,\, Q=Q^t,\, \mathrm{tr}Q=0,\,\lambda({Q})\in\Big(-{c\over2},{c\over2}\Big)\Big \},
\end{equation*}
where $\lambda(Q)$ represents the eigenvalues of $Q$. 

We consider the case where the rod-like molecules are located within a bounded 2D region $\Omega$, so that the mass conservation is assumed, 
\begin{align}
  \frac{1}{V(\Omega)}\int_{\Omega} c(\mr)\md\mr=c_0, \label{avc}
\end{align}
where $V(\Omega)$ represents the area of $\Omega$.

The dot product of two tensors is given by summing up the products of corresponding components, i.e. $U\cdot V=U_{ij}V_{ij}$. 
Hereafter, we adopt the convention of summation on repeated indices.
The differential operators can then be expressed as $\Delta=\partial_{ii}$, $\Delta^2=\partial_{iijj}$. 
If necessary, the identity tensor (matrix) is also written as $I_{ij}=\delta_{ij}$ using the Kronecker delta. 
The dot product between identical tensors is denoted as $|U|^2=U\cdot U$, and similarly $|\nabla u|^2=\partial_iu\partial_iu$, $|\nabla U|^2=\partial_iU_{jk}\partial_iU_{jk}$, etc. 
The inner product $(U,V)$ on $\Omega$ is defined as $\int_{\Omega}U\cdot V\md\mr$, and the $L^2$-norm is denoted as $\|U\|^2=(U,U)$.
When we mention other norms, we shall specify it in the subscript, such as $\|\cdot\|_{H^2(\Omega)}$.

\subsection{Free energy with quasi-entropy incorporating concentration variations}

The free energy derived from molecular theory consists of an entropy term and interaction terms, denoted as 
\begin{align}
  E[c,Q]=E_{\mathrm{ent}}+E_r. \nonumber
\end{align}
The interaction energy $E_r$ is obtained from gradient expansion.
Its complete form has been obtained for general rigid molecules in preceding works \cite{Han2015From}, where a rigid molecule may undergo any rotation in $SO(3)$. 
Although in this work the rod-like molecule is only allowed to rotate within a plane, the terms in the gradient expansion can be deduced in the same way, and the results are actually identical. 
The interaction energy contains bulk terms that do not involve spatial derivatives, and elastic terms that do involve them. 
To describe smectic structures, in elastic terms it would be necessary to include quadratic terms of spatial derivatives up to second-order. 
All the quadratic terms of first-order spatial derivatives are included, since they characterize the essential couplings between the orientational order and modulations (cf. Sec. V of \cite{Xu2020}). 
For the quadratic terms of second-order derivatives, we keep them to the minimally required, that is, only one term $|\Delta c|^2=\partial_{ii}c\partial_{jj}c$ is included to stabilize modulations (cf. discussions in \cite{Mei2015On}). 
Consequently, the interaction energy is written as 
\begin{align}
  E_r=&\,\int_{\Omega}c_{00}c^2+c_{02}\left|Q\right|^2+g_e\,\md\mr, \nonumber
\end{align}
where
\begin{align}
  g_e=& \, c_{20}\partial_k c\partial_k c+c_{21}\partial_k Q_{ij}\partial_k Q_{ij}+c_{22}\partial_i Q_{ij}\partial_k Q_{kj}+c_{23}\partial_j c\partial_i Q_{ij}
  +c_{40}\partial_{ii}c\partial_{jj}c. \label{elastic}
\end{align}
For the squared elastic terms, the coefficients $c_{40},\, c_{21},\, c_{22}>0$ stabilize the energy, while $c_{20}<0$ characterizes concentration modulations \cite{Mei2015On}. 

We turn to $E_{\mathrm{ent}}[c,Q]$.
Let us start from the entropy term given by the density function $f(\mr,\mm)$, 
\begin{align}
  \int_{\Omega\times \bS} f\log f\md\mr\md\mm+\int_{\Omega} (M-c)\log(1-c/M)\md\mr. \nonumber
\end{align}
where the second term describes a saturation value $c\le M$.
Separating the contribution of the orientational density \(\rho\) from the
entropy term, we rewrite the integrand with respect to \(\md\mr\) as
\begin{equation*}
	c\log c+(M-c)\log(1-c/M)
	+c\int_{\bS}\rho\log\rho\,\md\mm .
\end{equation*}
Notice that $\int_{\bS} \rho\log\rho\md\mm$ represents the entropy of the normalized orientational distribution. 
Hence, a reasonable structure of $E_{\mathrm{ent}}[c,Q]$ is
\begin{equation}
  c\log c+(M-c)\log(1-c/M)+c \zeta(Q/c) \nonumber
\end{equation}
for some function $\zeta$, where by definition $Q/c$ is the tensor averaged by $\rho$, which is exactly the tensor adopted for the nematic liquid crystals. 
One choice of $\zeta$ is to deduce from the maximum entropy state \cite{Han2015From,Mei2015On}.
It leads to a function with nice properties while implicitly defined through the normalized density function $\rho$. 
For rigid molecules allowing $SO(3)$ rotations, an elementary function satisfying these properties given by the maximum entropy state has been proposed as a substitution, called the quasi-entropy \cite{Xu2022Quasi}. 
The quasi-entropy is defined as the log-determinant of the covariance matrix.
For rod-like molecules in a plane, we consider the quasi-entropy in a similar way. 
The covariance matrix now reads
\begin{equation}
  \int_{\bS}\mm\otimes\mm \rho \,\md\mm-\left(\int_{\bS}\mm\rho\,\md\mm\right)\otimes \left(\int_{\bS}\mm\rho\,\md\mm\right). \nonumber
\end{equation}
Since $\rho(\mr,-\mm)=\rho(\mr,\mm)$, it holds $\int_{\bS}\mm\rho\,\md\mm=0$.
Therefore, the covariance matrix is exactly $\int_{\bS}\mm\otimes\mm \rho \md\mm=Q/c+I/2$.
The quasi-entropy is then given by 
\begin{align}
  q(Q/c)=-\log\det\left(\frac{Q}{c}+\frac{I}{2}\right)+\log\det\left(\frac{I}{2}\right). \label{qe_rho}
\end{align}
Let $\nu q$ take the place of $\zeta$ for some $\nu>0$. 
We finally arrive at $E_{\mathrm{ent}}=\int_{\Omega}\qeN(c,Q)\md\mr$, with 
\begin{equation}\label{quasi-entropy}
  \qeN(c,Q)=c\log c+(M-c)\log\left(1-{c\over M}\right)+c\nu q(Q/c). 
\end{equation}

The function \eqref{quasi-entropy} possesses several properties that are crucial for the structures of the gradient flow and its numerical schemes, which we discuss in what follows.
We begin with the function $q$ \eqref{qe_rho}, whose properties result from the log-determinant. 
\begin{lemma}\label{lemma:logdet}
  The function $-\log \det S$ is strictly convex w.r.t. positive definite $S\in \mathbb{R}^{n\times n}$. 
\end{lemma}
For the proof, see for example \cite{Wang2023Q}.

The function $q(S)$ is defined for any symmetric traceless tensor $S$ with its eigenvalues lying in $(-1/2,1/2)$. If the eigenvalues exceed this range, we define its value as $+\infty$. 
\begin{proposition}\label{prop:qeQ}
  The function $q(S)$ has the following properties. 
  \begin{enumerate}
  \item It is strictly convex w.r.t. $S$, with the unique minimizer $S=0$. 
  \item $q(S)$ is invariant under rotations: for $T\in \mathbb{R}^{2\times 2}$ satisfying $T^tT=I$, it holds $q(TST^t)=q(S)$. 
  \item It gives a barrier function in the sense that $\lim_{|\lambda(S)|\to (1/2)^-}q(S)=+\infty$.
  \end{enumerate}
\end{proposition}
\begin{proof}
  The convexity follows directly from Lemma \ref{lemma:logdet}. Denote by $\pm \lambda$ two eigenvalues of $S$. Since 
  \begin{equation}\nonumber
    \det \Big(S+\frac{I}{2}\Big)=\Big(\frac{1}{2}+\lambda\Big)\Big(\frac{1}{2}-\lambda\Big)\le \frac{1}{4},
  \end{equation}
  where the equality holds only when $\lambda=0$, the unique minimizer is $S=0$.

  The rotational invariance is recognized by the equality $\det(TAT^t)=\det T\det A\det T^t=\det A$.

  The third property is deduced by taking the limit $\lambda\to (1/2)^{-}$ in the equality above. 
\end{proof}

We then turn to the function \eqref{quasi-entropy}. 
Notice that the tensor $Q$ defined in this paper is proportional to the concentration $c$.
One fine consequence of this definition is that the interaction energy consists of quadratic terms. 
A seemingly drawback is that in $\qeN $ we have to deal with the term $q(Q/c)$ which may involve singular variables. 
But it turns out that defining $Q$ in this way actually leads to convexity of $\qeN $ that is crucial for the gradient flow and its numerical methods. 

Let us first clarify the domain of $\qeN $.
It is easy to see that $\qeN $ is well-defined and takes finite value in the set 
\begin{equation}
  \cQdomM=\{(c,Q)\in \cQdom | c\le M\}. 
\end{equation}
\begin{theorem}
  $\cQdomM$ is a bounded convex set. 
\end{theorem}
\begin{proof}
  The boundedness follows from the fact that each component in a symmetric matrix is controlled by its maximum and minimum eigenvalues. 

  For the convexity, choose two different elements $(c_1,Q_1),(c_2,Q_2)\in\cQdom $ with $c_1\le c_2$. 
  If $c_1=0$, we have $((c_1+c_2)/2,(Q_1+Q_2)/2)=(c_2/2,Q_2/2)$ and $\lambda(Q_2/c_2)\in (-1/2,1/2)$, so that $(c_2/2,Q_2/2)\in\cQdomM$. 
  If $c_1>0$, it holds that $Q_1/c_1+I/2$ and $Q_2/c_2+I/2$ are positive definite. Thus, their linear combination with positive coefficients, 
  \begin{equation*}
    {Q_1+Q_2\over c_1+c_2}+{I\over2}={c_1\over c_1+c_2}\left({Q_1\over c_1}+{I\over2}\right)+{c_2\over c_1+c_2}\left({Q_2\over c_2}+{I\over2}\right),
  \end{equation*}
  is also positive definite, yielding $({c_1+c_2\over2},{Q_1+Q_2\over2})\in\cQdomM $.
\end{proof}

It is worth noting that $\cQdomM$ is neither open nor closed.
A related consequence is that $\qeN $ is NOT continuous at $(c,Q)=(0,0)$. 
This can be recognized readily by examining
$$
Q=\mathrm{diag}\bigg(\frac{c}{2}-c\exp\Big(-\frac{1}{c^2}\Big),-\frac{c}{2}+c\exp\Big(-\frac{1}{c^2}\Big)\bigg). 
$$
When $c\to 0^+$, we have $\qeN (c,Q)\to +\infty$ but $\qeN (0,0)=0$. 
Nevertheless, we find that $\qeN $ still possesses several properties necessary in the analysis. 
To be prepared for our discussions afterwards, we introduce the interior and the closure of $\cQdomM$, given by 
\begin{align*}
  &(\cQdomM)^{\circ}=\Big \{(c,Q)\Big| 0 < c < M,\, Q=Q^t,\, \mathrm{tr}Q=0,\, \lambda(Q)\in\Big(-{c\over2},{c\over2}\Big)\Big \}.\\
  &\overline{\cQdomM}=\Big \{(c,Q)\Big|0\le c\le M,\, Q=Q^t,\, \mathrm{tr}Q=0,\, \lambda(Q)\in\Big[-{c\over2},{c\over2}\Big]\Big \}.
\end{align*}
Similar to the function $q$, in $\overline{\cQdomM}$ we let $\qeN $ take $+\infty$ when $|\lambda(Q)|=c/2>0$. 
Moreover, its partial derivatives,
\begin{equation*}
  \begin{aligned}
     & \frac{\partial\qeN}{\partial c}=\log\frac{Mc}{M-c}-\nu\left(\log\det\Big(\frac{Q}{c}+\frac{I}{2}\Big)-\log\det\Big(\frac{I}{2}\Big)\right)+\frac{\nu}{c}{\left(\frac{Q}{c}+\frac{I}{2}\right)}_{ij}^{-1}Q_{ij}, \\
     & \frac{\partial\qeN}{\partial Q_{ij}}=-\nu{\left(\frac{Q}{c}+\frac{I}{2}\right)}_{ij}^{-1},
  \end{aligned}
\end{equation*}
are only defined for $(c,Q)\in(\cQdomM)^{\circ} $.

\begin{theorem}\label{qe_properties}
  The quasi-entropy $\qeN (c,Q)$ has the following properties:
  \begin{enumerate}
  \item It is invariant under rotations: for any $2 \times 2$ matrix $T$ satisfying $TT^t=I$, it holds $\qeN(c,TQT^t)=\qeN(c,Q)$. 
  \item It is strictly convex on $\cQdomM$ with a finite lower bound. 
  \item It is lower semicontinuous on $\overline{\cQdomM}$. 
  \end{enumerate}
\end{theorem}
\begin{proof}
  The rotation invariance follows the same argument as the proof of Proposition \ref{prop:qeQ}. 

  The finite lower bound follows from the fact that $q$ has a finite lower bound. 

  For the strict convexity of $\qeN(c,Q)$, notice that $c\log c+(M-c)\log(1-c/M)$ is strictly convex w.r.t. $c$.
  We show that $\phi(c,Q)=c\nu q(Q/c)$ is convex w.r.t. $(c,Q)$. 
  Let $(c_1,Q_1),(c_2,Q_2)\in\cQdomM$.
  For $c_1=0$, $c_2>0$, we calculate directly that 
  \begin{align*}
    \frac{1}{2}\big(\phi(c_1,Q_1)+\phi(c_2,Q_2)\big)=\frac{1}{2} \phi(c_2,Q_2)
    =\phi\Big(\frac{c_2}{2},\frac{Q_2}{2}\Big)
    =\phi\Big(\frac{c_1+c_2}{2},\frac{Q_1+Q_2}{2}\Big). 
  \end{align*}
  If $c_1,c_2>0$, we deduce from the strict convexity of $q$ that 
  \begin{equation*}
    \begin{aligned}
      \phi\Big(\frac{c_1+c_2}{2},\frac{Q_1+Q_2}{2}\Big) =
      &\, {c_1+c_2\over2}\nu q\Big({Q_1+Q_2\over c_1+c_2}\Big)  \\
      =    &\, {c_1+c_2\over2}\nu q\Big({c_1\over c_1+c_2}\cdot{Q_1\over c_1}+{c_2\over c_1+c_2}\cdot{Q_2\over c_2}\Big)        \\
      \leq &\, {c_1+c_2\over2}\Bigg({c_1\over c_1+c_2}\nu q\Big({Q_1\over c_1}\Big)+{c_2\over c_1+c_2}\nu q\Big({Q_2\over c_2}\Big)\Bigg) \\
      =   &\,  {c_1\over2}\nu q\Big({Q_1\over c_1}\Big)+{c_2\over2}\nu q\Big({Q_2\over c_2}\Big)\\
      = & \, \frac{1}{2}\big(\phi(c_1,Q_1)+\phi(c_2,Q_2)\big),
    \end{aligned}
  \end{equation*}
  where the equality holds only when $Q_1/c_1=Q_2/c_2$.
  If $\qeN (c_1,Q_1)+\qeN (c_2,Q_2)=2\qeN ((c_1+c_2)/2,(Q_1+Q_2)/2)$, the term $c\log c+(M-c)\log(1-c/M)$ implies that $c_1=c_2$, so that $Q_1=Q_2$.
  Thus, the strict convexity of $\qeN $ is established. 

  To show the lower semicontinuity of $\qeN(c,Q)$, let $(c^{(k)},Q^{(k)})\to (\hat{c},\hat{Q})\in \overline{\cQdomM}$.
  Since $\qeN(c,Q)$ is continuous in $\cQdomM\backslash (0,0)$,
  we only need to consider two cases:
  \begin{enumerate}
  \item $(\hat{c},\hat{Q})=(0,0)$;
  \item $\hat{c}>0$, $|\lambda(\hat{Q})|=\hat{c}/2$, i.e. $\qeN(\hat{c},\hat{Q})=+\infty$. 
  \end{enumerate}
  For the former case, we deduce from the lower-boundedness of $q(S)\ge q(0)$ that
  $$
  \liminf_{k\to +\infty}c^{(k)}\nu q(Q^{(k)}/c^{(k)})\ge \liminf_{k\to +\infty} c^{(k)}\nu q(0)=0. 
  $$
  Together with the continuity of $c\log c+(M-c)\log(1-c/M)$ on $[0,M]$, we arrive at
  $$
  \liminf_{k\to +\infty}\qeN (c^{(k)},Q^{(k)})\ge 0=\qeN (0,0). 
  $$
  For the latter case, we have $\lim_{k\to +\infty}|\lambda(Q^{(k)}/c^{(k)})|=1/2$.
  Hence, by the third property in Proposition \ref{prop:qeQ},
  $$
  \lim_{k\to +\infty}c^{(k)}\nu q(Q^{(k)}/c^{(k)})=+\infty. 
  $$
  Notice again the continuity of $c\log c+(M-c)\log(1-c/M)$ on $[0,M]$ to conclude the proof. 
\end{proof}

\begin{corollary}\label{cor-q}
  For the quasi-entropy $\qeN(c,Q)$, it holds
  \begin{align}
    \nonumber
    & {\left(\frac{\partial\qeN(c_1,Q_1)}{\partial c_1}-\frac{\partial\qeN(c_2,Q_2)}{\partial c_2}\right)}(c_1-c_2)+{\left(\frac{\partial\qeN(c_1,Q_1)}{\partial Q_1}-\frac{\partial\qeN(c_2,Q_2)}{\partial Q_2}\right)}\cdot (Q_1-Q_2)\geq0, \\
    \nonumber
    & \frac{\partial\qeN(c_1,Q_1)}{\partial c_1}(c_1-c_2)+\frac{\partial\qeN(c_1,Q_1)}{\partial Q_1}\cdot (Q_1-Q_2)\geq\qeN(c_1,Q_1)-\qeN(c_2,Q_2),
  \end{align}
  where $(c_1,Q_1),(c_2,Q_2)\in (\cQdomM)^{\circ}$, $c_1,c_2>0$. The equalities hold only when $c_1=c_2$ and $Q_1=Q_2$.
\end{corollary}
\begin{proof}
  Define
  \begin{equation*}
    h(s)=\qeN\left(s(c_1,Q_1)+(1-s)(c_2,Q_2)\right).
  \end{equation*}
  When $(c_1,Q_1)\ne (c_2,Q_2)$, it follows that $h(t)$ is strictly convex since $\qeN$ is strictly convex w.r.t. $(c,Q)$.
  The two inequalities are deduced from $h'(1)> h'(0)$ and $h'(1)> h(1)-h(0)$, respectively, by the chain rule. 
\end{proof}

For convenience, we denote the bulk energy within the integral w.r.t. $\md\mr$ as 
\begin{equation}\label{bulk}
  g_b(c,Q)=\qeN(c,Q)+c_{00}c^2+c_{02}\left|Q\right|^2.
\end{equation}
Summarizing (\ref{bulk}) and (\ref{elastic}), the total free energy is given by 
\begin{equation}\label{E}
  E[c,Q]=\int_{\Omega}g_b(c,Q)+g_e(c,Q)\md\mr.
\end{equation}
In the above, a positive coefficient $\nu$ is introduced.
When the anisotropic part $q(Q/c)$ dominates in the entropy term, the choice of $\nu$ is actually not quite significant. 
This is because, after discarding the terms that depend only on \(c\), a
rescaling by \(c/c_0\) shows that the effective parameter is \(c_0/\nu\).

\subsection{Gradient Flow}

In the gradient flow, it is necessary to ensure that $c$ satisfies mass conservation and that $Q$ is symmetric.
To this end, we introduce projection operators below, 
\begin{align*}
  \scrG A=A-{1\over {V(\Omega)}}\int_\Omega A\mathrm{d}\mathbf{r},
  {(\scrP Q)}_{ij}={1\over2}(Q_{ij}+Q_{ji})-{1\over2}\delta_{ij}Q_{kk}. 
\end{align*}
For $\psi(\mr)$ satisfying $\int_{\Omega}\psi\,\md\mr=0$, it holds 
\begin{equation*}
  (\scrG \phi,\psi)=(\phi,\psi).
\end{equation*}
For any symmetric traceless tensor $S$, it holds 
\begin{equation*}
  Q\cdot S=\scrP Q\cdot S.
\end{equation*}

Following the idea to make the dissipation operator simple, we consider the gradient flow driven by the variational derivatives on which the above projection operators are imposed, 
\begin{align}
  &\frac{\partial c}{\partial t}=-\scrG \frac{\delta E}{\delta c},\label{gf_c}\\
   &\frac{\partial Q}{\partial t}=-\scrP \frac{\delta E}{\delta Q}.\label{gf_Q}
\end{align}
The gradient flow above needs to be supplemented by suitable boundary conditions. 
We shall consider Dirichlet type boundary conditions, giving prescribed values of $c$ and $Q$ on $\partial\Omega$, 
\begin{align}
  c=c_{\mathrm{bnd}},\quad Q=Q_{\mathrm{bnd}},\quad \text{on } \partial\Omega. \label{Dirichlet_bnd}
\end{align}
They shall be given such that there exists some $(\underline{c},\underline{Q})$ satisfying
\begin{align}
  \int_{\Omega}\underline{c}\md\mr=c_0,\quad (\underline{c},\underline{Q})\text{ a.e. in }\Omega,\quad E[\underline{c},\underline{Q}]<+\infty,\quad 
  (\underline{c},\underline{Q})=(c_{\mathrm{bnd}},Q_{\mathrm{bnd}}) \text{ on } \partial\Omega.\label{cond_bnd}
\end{align}
Moreover, the $c_{40}$ term in the free energy \eqref{E} requires additional boundary conditions on the normal derivatives of $c$.
If we adopt Dirichlet type boundary conditions, the normal derivative $\partial_{\bfn}c$ is also needed, where $\bfn =(n_1,n_2)$ denotes the outward unit normal vector. 
However, it turns out to be not easy to give reasonable values of $\partial_{\bfn}c$. 
So, we instead propose an evolution equation for $\partial_{\bfn} c$. 
The equation shall be able to maintain the energy dissipation dictated by the gradient flow. 
To find out the equation of $\partial_{\bfn}c$, let us revisit the calculations of the variational derivatives, 
\begin{align}
    \delta E= & \int_{\Omega}\left(\frac{\partial\qeN}{\partial c}+2c_{00}c\right)\delta c +2c_{20}\partial_ic\partial_i\delta c+c_{23}\partial_iQ_{ij}\partial_j\delta c+2c_{40}\partial_{ii}c\partial_{jj}\delta c \nonumber\\
    & \quad+\left(\frac{\partial\qeN}{\partial Q_{ij}}+2c_{02}Q_{ij}\right)\delta Q_{ij} +2c_{21}\partial_k Q_{ij}\partial_k \delta Q_{ij}+2c_{22}\partial_kQ_{kj}\partial_i\delta Q_{ij}+c_{23}\partial_ic \partial_j\delta Q_{ij}\md\mr \nonumber\\
    =& \int_{\Omega}\left(\frac{\partial\qeN}{\partial c}+2c_{00}c-2c_{20}\Delta c-c_{23}\partial_{ij}Q_{ij}+2c_{40}\Delta^2c\right)\delta c                                         \nonumber\\
              & \quad+\left(\frac{\partial\qeN}{\partial Q_{ij}}+2c_{02}Q_{ij}-2c_{21}\Delta Q_{ij}-2c_{22}\partial_{ik}Q_{kj}-c_{23}\partial_{ij}c\right)\delta Q_{ij} \md\mr                             \nonumber\\
              & +\int_{\partial \Omega}\big(2c_{20}\partial_{\bfn }c+c_{23}\partial_{i}Q_{ij}n_j-2c_{40}\partial_{\bfn }\Delta c\big)\delta c+2c_{40}\Delta c \delta ( \partial_\bfn c) \nonumber\\
              & \qquad +\big(2c_{21}\partial_{\bfn }Q_{ij}+2c_{22}\partial_k Q_{kj}n_i+c_{23}\partial_j cn_i\big)\delta Q_{ij}\md S. \label{delta_E}
\end{align}
Thus, the variational derivatives are given by 
\begin{align}
  &\frac{\delta E}{\delta c}=\frac{\partial\qeN}{\partial c}+2c_{00}c-2c_{20}\Delta c-c_{23}\partial_{ij}Q_{ij}+2c_{40}\Delta^2c,\label{var_c}\\
  &\frac{\delta E}{\delta Q_{ij}}=\frac{\partial\qeN}{\partial Q_{ij}}+2c_{02}Q_{ij}-2c_{21}\Delta Q_{ij}-2c_{22}\partial_{ik}Q_{kj}-c_{23}\partial_{ij}c. \label{var_Q}
\end{align}
On the boundary, both $\delta c$ and $\delta Q$ are zero because of the Dirichlet boundary conditions.
For $\partial_{\bfn}c$, the evolution equation is written down following the similar structure for $c$ and $Q$ in $\Omega$, that is, 
\begin{equation}\label{bndeq}
  \frac{\partial (\partial_{\bfn } c)}{\partial t}=-\Delta c,\quad \text{on }\partial\Omega.
\end{equation}
The gradient flow system is given by \eqref{gf_c}--\eqref{bndeq}.

\begin{theorem}
  The gradient flow system \eqref{gf_c}--\eqref{Dirichlet_bnd}, \eqref{var_c}--\eqref{bndeq} follows the energy dissipation law
  \begin{equation}
    \frac{\md}{\md t}E[c,Q]=-\left(\scrG\frac{\delta E}{\delta c},\scrG\frac{\delta E}{\delta c}\right)-\left(\scrP\frac{\delta E}{\delta Q_{ij}},\scrP\frac{\delta E}{\delta Q_{ij}}\right)-2c_{40}\int_{\partial\Omega}(\Delta c)^2\md S. \nonumber
  \end{equation}
\end{theorem}
\begin{proof}

  The derivation follows the same route as calculating the energy variation $\delta E$, just to replace $\delta \phi$ by $\partial \phi/\partial t$ for $\phi$ taking $E$, $c$, $Q$, and $\partial_{\bfn}c$.
  Notice that the Dirichlet boundary conditions \eqref{Dirichlet_bnd} imply that $\partial c/\partial t$, $\partial Q/\partial t$ are zero on $\partial\Omega$. 
  Then, substituting the time derivatives according to \eqref{gf_c}, \eqref{gf_Q} and \eqref{bndeq}, the energy dissipation law is established. 
\end{proof}

\begin{remark}
  If the free energy contains a boundary integral on $\partial\Omega$, we can follow the same route to construct a gradient flow involving these terms. 
\end{remark}

The gradient flow system \eqref{gf_c}--\eqref{Dirichlet_bnd}, \eqref{var_c}--\eqref{bndeq} is featured in two aspects: one is the coupled constraints for $c$ and $Q$ given by the domain $\cQdomM$ to ensure the terms involving $\qeN$ make sense, the other is the compliance of the boundary evolution equation \eqref{bndeq} with the energy dissipation law. 
It is crucial for numerical schemes to maintain these properties and structures, which we discuss in the next section.

\section{Numerical Methods}\label{NM}
To construct numerical schemes for the gradient flow system \eqref{gf_c}--\eqref{Dirichlet_bnd}, \eqref{var_c}--\eqref{bndeq}, we discuss time discretizations first, followed by spatial discretizations. 

In smectic liquid crystals, the spatial changing of the concentration is quite drastic.
As a result, it would be more probable for the numerical solutions to go out of $\cQdomM$ with simple, especially linearly implicit schemes.
Actually, we have encountered this problem frequently, which we will report in the following section. 
The emphasis of $(c,Q)\in \cQdomM$ urges us to deal with the derivatives of $\qeN$ implicitly in time. 
Since $\qeN$ is convex, it is then natural to adopt the idea of convex splitting for the time discretization, which also applies to the boundary evolution equations. 
It turns out that the presence of the boundary evolution equations indeed is compatible with the argument for establishing the theoretical results, including the existence and uniqueness in $\cQdomM$, the energy dissipation, as well as the error estimates. 
The convexity and lower semicontinuity of $\qeN$ turn out to play a key role in the derivations. 
For spatial discretizations, we write down a second order finite difference scheme with special care for its consistency with the integration by parts involving high order spatial derivatives and normal derivatives. 
In this way, the theoretical results for the time discretizations also hold for the full discretizations.

\subsection{Time Discretization}
When the normal derivative on the boundary is not prescribed, the convexity can still be recognized, as we will show below. 
\begin{proposition}\label{gamma_convex}
  When $\gamma\ge \big(\min\{c_{20}-{c_{23}^2\over4c_{22}},0\}\big)^2/(4c_{40})$, the functional 
  $$
  \int_{\Omega} g_e(c,Q)+\gamma c^2\md\mr, 
  $$
  with $g_e$ given in \eqref{elastic}, is convex w.r.t. $(c,Q)$ for $c\in H^2(\Omega)$, $c=c_{\mathrm{bnd}}$ on $\partial\Omega$, $Q\in H^1(\Omega)$. 
\end{proposition}
\begin{proof}
  For $4\gamma_1c_{22}\ge c_{23}^2$, it is easy to verify that 
  \begin{equation}
    c_{22}\partial_jQ_{ij}\partial_kQ_{ik}+c_{23}\partial_ic\partial_jQ_{ij}+\gamma_1\partial_ic\partial_ic\ge 0. \nonumber
  \end{equation}
  Thus, we write 
  \begin{equation}\label{g_e}
    \begin{aligned}
      \int_\Omega g_e(c,Q)+\gamma c^2\md\mr
      = & \int_\Omega c_{22}\partial_jQ_{ij}\partial_kQ_{ik}+c_{23}\partial_ic\partial_jQ_{ij}+\gamma_1\partial_ic\partial_ic
      \\
      & +c_{40}\partial_{ii}c\partial_{jj}c+(c_{20}-\gamma_1)\partial_ic\partial_ic+\gamma c^2\md\mr.
    \end{aligned}
  \end{equation}  
  If $c_{20}\ge \gamma_1$, the proof is done since the functional is quadratic and nonnegative.
  If $c_{20}<\gamma_1$, we utilize 
  \begin{equation}
    \int_{\Omega}\partial_ic\partial_ic\md\mr
    =\int_{\partial\Omega}c\partial_{\bfn}c\md S-\int_{\Omega}c\partial_{ii}c\md\mr, \nonumber
  \end{equation}
  to write the second line as 
  \begin{equation}\label{g_e_part2}
    \int_\Omega c_{40}\partial_{ii}c\partial_{jj}c-(c_{20}-\gamma_1)c\partial_{ii}c+\gamma c^2\md\mr
    +(c_{20}-\gamma_1)\int_{\partial\Omega}c_{\mathrm{bnd}}\cdot\partial_{\bfn }c\md S.
  \end{equation}
  For $4\gamma c_{40}\ge (c_{20}-\gamma_1)^2$, it holds 
  \begin{equation}
    c_{40}\partial_{ii}c\partial_{jj}c-(c_{20}-\gamma_1)c\partial_{ii}c+\gamma c^2\ge 0. \nonumber
  \end{equation}
  We complete the proof noting that the boundary integral is linear w.r.t. $c$. 
\end{proof}

Together with the convexity of the quasi-entropy $\qeN$, the free energy can be expressed as the difference of two convex functionals, 
\begin{equation}\label{splitting}
  E[c,Q]=\int_\Omega \qeN(c,Q)+g_+(c,Q)-g_-(c,Q)\md\mr,
\end{equation}
where
\begin{equation}\label{g^+-}
  \begin{aligned}
    g_+(c,Q) & =g_e(c,Q)+\gamma c^2,\\
    g_-(c,Q) & =(\gamma-c_{00})c^2-c_{02}Q^2,
  \end{aligned}
\end{equation}
and $\gamma\ge c_{00}$ is chosen according to the above lemma. 

Let $(c^n,Q^n)$ denote the numerical solution at $t^n=n\delta t$.
In the spirit described in the beginning of this section, we construct the scheme below, 
\begin{align}
  & \frac{c^{n+1}-c^n}{\delta t}=- \scrG \bigg( \frac{\partial \qeN(c^{n+1},Q^{n+1})}{\partial c^{n+1}}+2(c_{00}-\gamma)c^{n}+2\gamma c^{n+1}-2c_{20}\Delta c^{n+1}                         \nonumber\\
  & \qquad\qquad  -c_{23}\partial_{ij}Q_{ij}^{n+1}+2c_{40}\Delta^2c^{n+1}\bigg) \label{t_discrete_c}\\
  & \frac{Q_{ij}^{n+1}-Q_{ij}^n}{\delta t}=-  \scrP \Bigg( {\left(\frac{\partial \qeN(c^{n+1},Q^{n+1})}{\partial Q^{n+1}}\right)}_{ij}+2c_{02}Q^n_{ij}-2c_{21}\Delta Q_{ij}^{n+1} \nonumber\\
  & \qquad\qquad-2c_{22}\partial_{ik}Q_{kj}^{n+1}-c_{23}\partial_{ij}c^{n+1}\Bigg),\label{t_discrete_Q}\\
  & \frac{\partial_{\bfn }c^{n+1}-\partial_{\bfn }c^n}{\delta t}=-\Delta c^{n+1},\quad \mathrm{on\,\partial \Omega}, \label{t_discrete_bnd}
\end{align}
with the boundary values of $c^{n+1}$ and $Q^{n+1}$ given by \eqref{Dirichlet_bnd}.

In the following, we shall establish the existence and uniqueness, energy dissipation and error estimate for the scheme \eqref{t_discrete_c}--\eqref{t_discrete_bnd}.

\begin{theorem}\label{energylaw}
  Assume that $\partial \Omega$ is smooth. For any $\delta t>0$, the scheme \eqref{t_discrete_c}--\eqref{t_discrete_bnd} has the properties below:
  \begin{enumerate}
  \item It has a unique solution $c^{n+1}\in H^2(\Omega)$, $Q^{n+1}\in H^1(\Omega)$ and $(c^{n+1},Q^{n+1})\in (\cQdomM)^{\circ}$ a.e. in $\Omega $.
  \item It is mass conservative, i.e. $\int_{\Omega}c^{n+1}\md\mr=\int_{\Omega}c^n\md\mr$.
  \item It satisfies the energy dissipation law,
    \begin{align}
      & E[c^{n+1},Q^{n+1}]+\frac{1-(c_{00}-\gamma)\delta t}{\delta t}\|c^{n+1}-c^n\|^2+\frac{1-c_{02}\delta t}{\delta t}(\|Q^{n+1}-Q^n\|^2)   \nonumber\\
      & \quad +\int_\Omega g_+(c^{n+1}-c^n,Q^{n+1}-Q^n)\md\mr+\frac{2c_{40}}{\delta t}\int_{\partial\Omega}(\partial_{\bfn }c^{n+1}-\partial_{\bfn }c^n)^2\md S\leq E[c^{n},Q^{n}]. \label{energy_law}
    \end{align}
  \end{enumerate}
\end{theorem}
\begin{proof}
  We first discuss the existence, which is established by recognizing the solution to the scheme as the minimizer of a convex functional. 
  For the presence of the boundary normal derivative $\partial_{\bfn}c^{n+1}$, it is necessary to include a boundary integral in such a functional, written as
  \begin{align}
    F[ c^{n+1},Q^{n+1}]=\int_{\Omega}\qeN(c^{n+1},Q^{n+1})\md\mr+F_1[  c^{n+1},Q^{n+1}],\nonumber
  \end{align}
  where
  \begin{align}
    F_1[ & c^{n+1},Q^{n+1}]=\frac{1}{2\delta t}\Big(\|c^{n+1}-c^n\|^2+\|Q^{n+1}-Q^n\|^2
    +2c_{40}\int_{\partial \Omega}(\partial_{\bfn}c^{n+1}-\partial_{\bfn}c^n)^2\md S\Big)    \nonumber\\
    & +\int_{\Omega}g_+(c^{n+1},Q^{n+1})+2(c_{00}-\gamma)c^n\cdot c^{n+1}+2c_{02}Q^n\cdot Q^{n+1} \md\mr.\nonumber
  \end{align}

  Theorem \ref{qe_properties} and Proposition \ref{gamma_convex} guarantee the convexity of $\int_{\Omega}(\qeN+g_+)\md\mr$. 
  So, we can verify that the functional $F[c^{n+1},Q^{n+1}]$ given above is strictly convex because the remaining are squared and linear terms.
  We further show that the functional $F[c^{n+1},Q^{n+1}]$ is bounded from below.
  The proof of Proposition \ref{gamma_convex} (cf. \eqref{g_e_part2}) implies that $\int_{\Omega}g_+(c,Q)\md\mr\ge (c_{20}-\gamma_1)\int_{\partial\Omega}c_{\mathrm{bnd}}\partial_{\bfn}c\md\mr$, which can be controlled by the squared term in the boundary integral. Together with Theorem \ref{qe_properties} giving the lower-boundedness of $\qeN$, and the fact that $c^n$ and $Q^n$ are constrained in the bounded domain $\overline{\cQdomM}$, the lower bound of the functional $F[c^{n+1},Q^{n+1}]$ is recognized. 
 
  We show that there exists a unique minimizer of the functional $F$ under the constraint $\int_\Omega\,c^{n+1}\md\mr=\int_\Omega\,c^n\md\mr$.
  Recall that we have assumed in \eqref{cond_bnd} when specifying the boundary values that there exists some $(\underline{c},\underline{Q})$ such that $E[\underline{c},\underline{Q}]<+\infty $.
  It is easy to show that $F[\underline{c},\underline{Q}]-E[\underline{c},\underline{Q}]\le A(\|\underline{c}\|^2+\|\underline{Q}\|^2+\|c^n\|^2+\|Q^n\|^2)$ for some $A>0$, so that $F[\underline{c},\underline{Q}]<+\infty$. 
  Therefore, we can choose a sequence $(c^{(k)},Q^{(k)})$ satisfying $\int_\Omega\,c^{(k)}\md\mr=\int_\Omega\,c^n\md\mr$ and
  \begin{equation*}
    \lim_{k\rightarrow +\infty} F[c^{(k)},Q^{(k)}]=\inf F[c^{n+1},Q^{n+1}]<+\infty.
  \end{equation*}
  The functional $F$ ensures that $\|Q^{(k)}\|_{H^1}$ and $\|\Delta c^{(k)}\|_{L^2}$ are bounded, with the latter implying $ \|c^{(k)}\|_{H^2}$ bounded using the regularity of Poisson's equation.
  In addition, the fact that $H^1(\Omega)$ is a compact embedding into $L^2(\partial\Omega)$ implies that $\|\partial_{\bfn}c\|_{L^2(\partial\Omega)}$ is controlled by $\|c\|_{H^2(\Omega)}$. 
  Thus, we could choose a subsequence, still denotes as $(c^{(k)},Q^{(k)})$, such that $ \{c^{(k)}\} $ converges $H^2$-weakly to $\bar{c}$, and $ \{Q^{(k)}\} $ converges $H^1$-weakly to $\bar{Q}$, respectively. Then we have
  \begin{equation*}
    \lim_{k\rightarrow+\infty}F_1[c^{(k)},Q^{(k)}]=F_1[\bar{c},\bar{Q}],
  \end{equation*}
  We further choose an a.e. convergent subsequence by the Riesz theorem. 
 Now, it is necessary to use the lower semicontinuity of $\qeN$ on $\overline{\cQdomM}$,  followed by the Fatou's lemma, to arrive at 
  \begin{equation*}
    \int_\Omega\qeN(\bar{c},\bar{Q})\md\mr\leq\int_\Omega\liminf_{k\rightarrow+\infty}\qeN(c^{(k)},Q^{(k)})\md\mr\leq\liminf_{k\rightarrow+\infty}\int_\Omega\qeN(c^{(k)},Q^{(k)})\md\mr.
  \end{equation*}
  Hence, 
  \begin{equation*}
    \begin{aligned}
      F[\bar{c},\bar{Q}]
      \leq\lim_{k\rightarrow+\infty}F_1[c^{(k)},Q^{(k)}]+\liminf_{k\rightarrow+\infty}\int_{\Omega}\qeN(c^{(k)},Q^{(k)})\md\mr
      \leq \inf F[c^{n+1},Q^{n+1}],
    \end{aligned}
  \end{equation*}
  which implies that $(\bar{c},\bar{Q})$ is the minimizer of $F[c^{n+1},Q^{n+1}]$. Its uniqueness could be obtained directly from the strict convexity of $F[c^{n+1},Q^{n+1}]$. 
  
  Next, we show that $(\bar{c},\bar{Q})$ is a.e. in $(\cQdomM)^{\circ}$.
  Since the minimizing energy is finite, the set on which
  \(q(\bar Q/\bar c)=+\infty\) has measure zero. Hence
  \(q(\bar Q/\bar c)<+\infty\) a.e. on \(\{0<\bar c<M\}\).
  Define 
  \begin{equation*}
    \Omega_1=\{\mr\in\Omega|\bar{c}(\mr)=0\},\quad \Omega_2=\{\mr\in\Omega|\bar{c}(\mr)=M\}, 
  \end{equation*}
  with their measures denoted by $|\Omega_1|,\,|\Omega_2|$.
  Assume that $\xi>0$ and let
  \begin{equation*}
    (\tilde{c},\tilde{Q})=\big(\xi \underline{c}+(1-\xi)\bar{c},\xi \underline{Q}+(1-\xi)\bar{Q}\big),
  \end{equation*}
  According to our discussions and assumptions above, it holds that 
  $F_1[\underline{c},\underline{Q}]$, $F_1[\bar{c},\bar{Q}]$, 
  $\int_{\Omega}\qeN(\underline{c},\underline{Q})\md\mr$, and $\int_{\Omega}\qeN(\bar{c},\bar{Q})\md\mr$ all take some finite values.
  Below, we shall use $A_i$ to denote some finite constants. 
  By convexity, we deduce that 
  \begin{align*}
    F_1[\tilde{c},\tilde{Q}]-F_1[\bar{c},\bar{Q}]\le \xi \Big(F_1[\underline{c},\underline{Q}]-F_1[\bar{c},\bar{Q}]\Big)=A_1\xi. 
  \end{align*}
  and 
  \begin{align*}
    \int_{\Omega\backslash(\Omega_1\cup\Omega_2)}\qeN(\tilde{c},\tilde{Q})-\qeN(\bar{c},\bar{Q})\md\mr
    \le \xi\int_{\Omega\backslash(\Omega_1\cup\Omega_2)}\qeN(\underline{c},\underline{Q})-\qeN(\bar{c},\bar{Q})\md\mr=A_2\xi. 
  \end{align*}
  For the region $\Omega_1\cup\Omega_2$, we calculate that 
  \begin{align*}
    \int_{\Omega_1}\qeN(\tilde{c},\tilde{Q})-\qeN(\bar{c},\bar{Q})\md\mr
    =&\,\int_{\Omega_1}\xi \underline{c}\log (\xi\underline{c})+(M-\xi\underline{c})\log(1-\xi\underline{c}/M)+\xi\underline{c}\nu q(\underline{Q}/\underline{c})\md\mr \\
    \leq&\,\xi\left(\log(\xi M)\int_{\Omega_1}\underline{c}\md\mr+\int_{\Omega_1}(M-\underline{c})\log(1-\underline{c}/M)+\underline{c}\nu q(\underline{Q}/\underline{c})\md\mr\right)\\
    \leq&\,\xi\left(\log(\xi M)\int_{\Omega_1}\underline{c}\md\mr+A_3\right),
  \end{align*}
and
\begin{align*}
    \int_{\Omega_2}\qeN(\tilde{c},\tilde{Q})-\qeN(\bar{c},\bar{Q})\md\mr
    \leq&\,\int_{\Omega_2}(\tilde{c}-\bar{c})\log\frac{M-\tilde{c}}{M\tilde{c}}+\tilde{c}\nu q(\tilde{Q}/\tilde{c})-\bar{c}\nu q(\bar{Q}/\bar{c})\md\mr \nonumber\\
    \leq&\,\xi\log(\xi M)\int_{\Omega_2}(M-\underline{c})\md\mr+\int_{\Omega_2}\tilde{c}\nu q(\tilde{Q}/\tilde{c})-M\nu q(\bar{Q}/M)\md\mr. 
  \end{align*}
  Recall that $\phi(c,Q)=c\nu q(Q/c)$, defined in the proof of Theorem \ref{qe_properties}, is also convex.
  We thus obtain 
  \begin{align*}
    \int_{\Omega_2}\tilde{c}\nu q(\tilde{Q}/\tilde{c})-M\nu q(\bar{Q}/M)\md\mr
    \le \xi \int_{\Omega_2}\Big(\underline{c}\nu q(\underline{Q}/\underline{c})-M\nu q(\bar{Q}/M)\Big)\md\mr=A_4\xi. 
  \end{align*}
  Note that
  $$
  A_0\triangleq\int_{\Omega_1}\underline{c}\md\mr+\int_{\Omega_2}(M-\underline{c})\md\mr\ge 0.
  $$
  Combining the estimates above, we arrive at 
  \begin{align*}
    F[\tilde{c},\tilde{Q}]-F[\bar{c},\bar{Q}]\le \xi(A_0\log(\xi M)+A_1+A_2+A_3+A_4). 
  \end{align*}
  If $A_0>0$, there exists sufficiently small $\xi>0$ to make the right-hand side negative, leading to a contradiction. 
  Hence $A_0=0$, which, by $0<\underline{c}<M$ a.e. in $\Omega$, implies $|\Omega_1|=|\Omega_2|=0$.
  By calculating the variation of $F[c^{n+1},Q^{n+1}]$ similar to \eqref{delta_E}, the minimizer of $F$ gives a solution to the scheme \eqref{t_discrete_c}--\eqref{t_discrete_bnd}.

  We now turn to uniqueness, since it is not guaranteed that every solution of the scheme is a minimizer of $F[c^{n+1},Q^{n+1}]$.
  Suppose that the scheme has two solutions $(c^{n+1,1},Q^{n+1,1})$ and $(c^{n+1,2},Q^{n+1,2})$.
  Denote $\alpha=c^{n+1,1}-c^{n+1,2}$, $U=Q^{n+1,1}-Q^{n+1,2}$.
  They obey the following equations, 
  \begin{align}
     & \frac{\alpha}{\delta t}=-\scrG \Big(\frac{\partial \qeN\left(c^{n+1,1},Q^{n+1,1}\right)}{\partial c^{n+1,1}}-\frac{\partial \qeN\left(c^{n+1,2},Q^{n+1,2}\right)}{\partial c^{n+1,2}}+2\gamma \alpha-2c_{20}\Delta \alpha \nonumber\\
     & \qquad\qquad\qquad -c_{23}\partial_{ij}U_{ij}+2c_{40}\Delta^2 \alpha\Big), \nonumber\\
     & \frac{U_{ij}}{\delta t}=-\scrP \Bigg({\left(\frac{\partial \qeN(c^{n+1,1},Q^{n+1,1})}{\partial Q^{n+1,1}}\right)}_{ij}-{\left(\frac{\partial \qeN(c^{n+1,2},Q^{n+1,2})}{\partial Q^{n+1,2}}\right)}_{ij}-2c_{21}\Delta U_{ij} \nonumber\\
     & \qquad\qquad\qquad-2c_{22}\partial_{ik}U_{kj}-c_{23}\partial_{ij}\alpha\Bigg),\nonumber\\
     & \frac{\partial_{\bfn }\alpha}{\delta t}=-\Delta \alpha,\quad \alpha=U_{ij}=0,\quad \text{ on }\partial\Omega. \nonumber
  \end{align}
  Taking the inner product of these equations with $\alpha,\,U,\,\partial_{\bfn}d$, respectively, we obtain 
  \begin{equation}\label{RS}
    \begin{aligned}
      \frac{1}{\delta t} & (\|\alpha\|^2+\|U\|^2)+\int_\Omega \left(\frac{\partial \qeN\left(c^{n+1,1},Q^{n+1,1}\right)}{\partial c^{n+1,1}}-\frac{\partial \qeN\left(c^{n+1,2},Q^{n+1,2}\right)}{\partial c^{n+1,2}}\right)\alpha \\
      & +\left(\frac{\partial \qeN\left(c^{n+1,1},Q^{n+1,1}\right)}{\partial Q^{n+1,1}}-\frac{\partial \qeN\left(c^{n+1,2},Q^{n+1,2}\right)}{\partial Q^{n+1,2}}\right)\cdot U+2g_+(\alpha,U)\md\mr \\
      & +2c_{40}\int_{\partial\Omega}\frac{(\partial_{\bfn}\alpha)^2}{\delta t}\md S=0.
    \end{aligned}
  \end{equation}
  As $\left(c^{n+1,1},Q^{n+1,1}\right),\left(c^{n+1,2},Q^{n+1,2}\right)\in \mathcal{Q}_\mathrm{phys}$ are a.e. in $\Omega $, we have
  \begin{equation*}
    \begin{aligned}
       & \left(\frac{\partial \qeN\left(c^{n+1,1},Q^{n+1,1}\right)}{\partial c^{n+1,1}}-\frac{\partial \qeN\left(c^{n+1,2},Q^{n+1,2}\right)}{\partial c^{n+1,2}}\right)\alpha \\
       & +\left(\frac{\partial \qeN\left(c^{n+1,1},Q^{n+1,1}\right)}{\partial Q^{n+1,1}}-\frac{\partial \qeN\left(c^{n+1,2},Q^{n+1,2}\right)}{\partial Q^{n+1,2}}\right)\cdot U\geq0,\quad \mathrm{a.e.\, in\,\,\Omega},
    \end{aligned}
  \end{equation*}
  according to Corollary \ref{cor-q}. Thus, the left-hand side of (\ref{RS}) is nonnegative, which implies that $\alpha=0,\,U=0$ a.e. in $\Omega$ and $\partial_{\bfn}\alpha=0$ a.e. in $\partial\Omega$.

  We turn to the energy dissipation. 
  Taking the inner product of the \eqref{t_discrete_c}--\eqref{t_discrete_bnd} with $(c^{n+1}-c^n)$, $(Q^{n+1}-Q^n)$, $\partial_{\bfn}c^{n+1}-\partial_{\bfn}c^n$, respectively, we deduce that 
  \begin{equation}\nonumber
    \begin{aligned}
      \frac{\|c^{n+1}-c^n\|^2}{\delta t}= 
      & -\left(\frac{\partial \qeN(c^{n+1},Q^{n+1})}{\partial c^{n+1}},c^{n+1}-c^n\right)-2((c_{00}-\gamma)c^{n}+\gamma c^{n+1},c^{n+1}-c^n) \\
      & +2c_{20}(\Delta c^{n+1},c^{n+1}-c^n)-2c_{40}(\Delta^2c^{n+1},c^{n+1}-c^n)                                                     \\
      & +c_{23}(\partial_{ij}Q_{ij}^{n+1},c^{n+1}-c^n),                                                                                \\
      \frac{\|Q^{n+1}-Q^n\|^2}{\delta t}= 
      & -\left(\frac{\partial \qeN(c^{n+1},Q^{n+1})}{\partial Q^{n+1}},Q^{n+1}-Q^n\right)-2c_{02}(Q^{n},Q^{n+1}-Q^n)               \\
      & +2c_{21}\left(\Delta Q^{n+1},Q^{n+1}-Q^n\right)+2c_{22}(\partial_{ik}Q_{kj}^{n+1},Q_{ij}^{n+1}-Q_{ij}^n)                      \\
      & +c_{23}(\partial_{ij}c^{n+1},Q_{ij}^{n+1}-Q_{ij}^n),\\
      \frac{1}{\delta t}\int_{\partial\Omega} (\partial_{\bfn }c^{n+1}-\partial_{\bfn }c^n)^2\md S =&
      -\int_{\partial\Omega}(\partial_{\bfn }c^{n+1}-\partial_{\bfn }c^n)\Delta c^{n+1}\md S. 
    \end{aligned}
  \end{equation}
  According to Corollary \ref{cor-q},
  \begin{equation*}
    \begin{aligned}
      & \frac{\partial\qeN(c^{n+1},Q^{n+1})}{\partial c^{n+1}}(c^{n+1}-c^n)+\frac{\partial\qeN(c^{n+1},Q^{n+1})}{\partial Q^{n+1}}(Q^{n+1}-Q^n) \\
      \geq & \,\qeN(c^{n+1},Q^{n+1})-\qeN(c^n,Q^n).
    \end{aligned}
  \end{equation*}
  For the other terms, we utilize the identity 
  $$
  u_1(v_1-v_2)+v_1(u_1-u_2)=u_1v_1-u_2v_2+(u_1-u_2)(v_1-v_2). 
  $$
  Note that 
  $$
  (\Delta^2 c^{n+1},c^{n+1}-c^n)=\left(\Delta c^{n+1},\Delta(c^{n+1}-c^n)\right)-\int_{\partial\Omega}\Delta c^{n+1}\partial_{\bfn}(c^{n+1}-c^n)\md S. 
  $$
  Summing up the above three equalities, we derive the energy law \eqref{energy_law}. 
  Note that $\int_\Omega g_+(c^{n+1}-c^n,Q^{n+1}-Q^n)\md\mr\geq0$ because $c^{n+1}-c^n=0$ and $Q^{n+1}-Q^n=0$ on $\partial\Omega $.
\end{proof}

The error estimate of the scheme \eqref{t_discrete_c}--\eqref{t_discrete_bnd} involves the normal derivative.
As we will see below, the propagation of error is not broken by the boundary integral emerged from integration by parts because the coupling term is cancelled. 
Define $\errc^n(\mr)=c(\mr,t^n)-c^n(\mr)$, $R^n(\mr)=Q(\mr,t^n)-Q^n(\mr)$, and
$\varepsilon^{n}=\|\errc^{n}\|^2+2c_{40}\int_{\partial\Omega}{(\partial_\bfn \errc ^{n})}^2\md S+\|R^{n}\|^2$.

\begin{theorem}
  Assume that $\partial_{tt}c,\partial_{tt}Q,\partial_{t}c,\partial_{t}Q\in L^2(0, T;L^2(\Omega))$, $\partial_{tt}\partial_{\bfn }c\in L^2\big(0,T;L^2(\partial\Omega)\big)$. For the scheme \eqref{t_discrete_c}--\eqref{t_discrete_bnd}, it holds
  \begin{equation}\label{Error}
    \begin{aligned}
      \varepsilon^n\leq C\exp\big({(1-C\delta t)}^{-1}t^n\big)
       & \delta t^2\cdot\int_0^{t^n}\Big(\|\partial_{tt}c\|^2+\|\partial_{tt}Q\|^2                                            \\
       & +\|\partial_{t}c\|^2+\|\partial_{t}Q\|^2+\int_{\partial\Omega}{(\partial_{tt}\partial_{\bfn }c)}^2\md S\Big)\md t,
    \end{aligned}
  \end{equation}
  where the constant $ C $ depends on $c_{00}-\gamma$, $c_{02}$ and $c_{40}$.
\end{theorem}

\begin{proof}
  We deduce from \eqref{t_discrete_c}--\eqref{t_discrete_bnd} the equations for the error, 
  \begin{align}
    & {\errc ^{n+1}-\errc ^n\over\delta t}+T_c^n=-\scrG \Big({\partial \qeN\left(c(\mr,t^{n+1}),Q(\mr,t^{n+1})\right)\over \partial c(\mr,t^{n+1})}-{\partial \qeN\left(c^{n+1}(\mr),Q^{n+1}(\mr)\right)\over \partial c^{n+1}(\mr)}   \nonumber\\
    & \quad+2(c_{00}-\gamma)(\errc ^{n}+\theta_c^n)+2\gamma \errc ^{n+1}-2c_{20}\Delta \errc ^{n+1}-c_{23}\partial_{ij}R_{ij}^{n+1}+2c_{40}\Delta^2\errc ^{n+1}\Big), \nonumber\\
    & {R_{ij}^{n+1}-R_{ij}^n\over\delta t}+T^n_{ij}=-\scrP \Bigg({\displaystyle {\left({\partial \qeN\left(c(\mr,t^{n+1}),Q(\mr,t^{n+1})\right)\over \partial Q(\mr,t^{n+1})}\right)}_{ij}}-{\displaystyle {\left({\partial \qeN\left(c^{n+1}(\mr),Q^{n+1}(\mr)\right)\over \partial Q^{n+1}(\mr)}\right)}_{ij}} \nonumber\\
    & \quad+2c_{02}(R_{ij}^n+\Theta^n_{ij})-2c_{21}\Delta R_{ij}^{n+1}-2c_{22}\partial_{ik}R_{kj}^{n+1}-c_{23}\partial_{ij}\errc ^{n+1}\Bigg), \nonumber\\
    & {\partial_{\bfn }\errc ^{n+1}-\partial_{\bfn }\errc ^{n}\over\delta t}+T_{\partial\Omega}^n=-\Delta \errc ^{n+1},\quad \errc ^{n+1}=0,\quad R^{n+1}=0,\quad \text{ on } \partial \Omega. \nonumber
  \end{align}
  where the truncation errors are
  \begin{equation}\nonumber
    \begin{aligned}
      T_c^n                & ={1\over\delta t}\big(c(\mr,t^{n+1})-c(\mr,t^n)\big)-\partial_t c(\mr,t^{n+1})={1\over\delta t}\int_{t^n}^{t^{n+1}}(t^n-s)\partial_{tt}c\,\md t, \\
      \theta_c^n                & =c(\mr,t^{n+1})-c(\mr,t^n)=\int_{t^n}^{t^{n+1}}\partial_t c\,\md t,                                                                              \\
      T^n                  & ={1\over\delta t}\int_{t^n}^{t^{n+1}}(t^n-s)\partial_{tt}Q\,\md t,                                                                               \\
      \Theta^n                  & =Q(\mr,t^{n+1})-Q(\mr,t^n)=\int_{t^n}^{t^{n+1}}\partial_t Q\,\md t,                                                                              \\
      T_{\partial\Omega}^n & ={1\over\delta t}\left(\partial_{\bfn }c(\mr,t^{n+1})-\partial_{\bfn }c(\mr,t^n)\right)-\partial_t\partial_{\bfn }c(\mr,t^{n+1})                         \\
                           & ={1\over\delta t}\int_{t^n}^{t^{n+1}}(t^n-s)\partial_{tt}\partial_{\bfn }c\,\md t.
    \end{aligned}
  \end{equation}
  Take the inner product of the three equations above with $\errc ^{n+1}$, $R^{n+1}$, $\partial_{\bfn}\errc^{n+1}$, respectively and sum them up. Notice that $\errc ^{n+1}=\scrG \errc ^{n+1}$ and $R^{n+1}=\scrP R^{n+1}$, so that the projection operators do not affect.
  We only emphasize the derivations $\qeN$ and the boundary terms. 
  For the $\qeN$ terms, we deduce from Corollary \ref{cor-q} that
  \begin{equation*}
    \begin{aligned}
        & \left({\partial \qeN\left(c(\mr,t^{n+1}),Q(\mr,t^{n+1})\right)\over \partial c(\mr,t^{n+1})}-{\partial \qeN\left(c^{n+1}(\mr),Q^{n+1}(\mr)\right)\over \partial c^{n+1}(\mr)}\right)\errc ^{n+1}      \\
      + & \left({\partial \qeN\left(c(\mr,t^{n+1}),Q(\mr,t^{n+1})\right)\over \partial Q(\mr,t^{n+1})}-{\partial \qeN\left(c^{n+1}(\mr),Q^{n+1}(\mr)\right)\over \partial Q^{n+1}(\mr)}\right)\cdot R^{n+1}\geq0. 
    \end{aligned}
  \end{equation*}
  For the boundary terms, we have 
  \begin{equation*}
    \begin{aligned}
      & \int_\Omega\Big(-2\gamma \errc ^{n+1}+2c_{20}\Delta \errc ^{n+1}+c_{23}\partial_{ij}R_{ij}^{n+1}-2c_{40}\Delta^2\errc ^{n+1}\Big) \errc ^{n+1}                 \\
      & \qquad+\left(2c_{21}\Delta R_{ij}^{n+1}+2c_{22}\partial_{ik}R_{kj}^{n+1}+c_{23}\partial_{ij}\errc ^{n+1}\right)R_{ij}^{n+1}\md\mr                    \\
      =    & -\int_\Omega 2g_e(\errc ^{n+1},R^{n+1})+2\gamma {(\errc ^{n+1})}^2\md\mr+2c_{40}\int_{\partial\Omega}\Delta \errc ^{n+1}\partial_\bfn \errc ^{n+1}\md S, 
    \end{aligned}
  \end{equation*}
  with 
  \begin{equation*}
    \begin{aligned}
           & \int_{\partial\Omega}\Delta \errc ^{n+1}\partial_\bfn \errc ^{n+1}\md S                                                                                                                                             \\
      =    & \int_{\partial\Omega}\big(-{\partial_\bfn  \errc ^{n+1}-\partial_\bfn  \errc ^n\over\delta t}-T^n_{\partial\Omega}\big)\partial_\bfn \errc ^{n+1}\md S                                                       \\
      \leq & {1\over2\delta t}\int_{\partial\Omega}{(\partial_\bfn  \errc ^n)}^2-{(\partial_\bfn  \errc ^{n+1})}^2\md S+{1\over2}\int_{\partial\Omega}{(T_{\partial\Omega}^n)}^2+{(\partial_\bfn  \errc ^{n+1})}^2\md S.
    \end{aligned}
  \end{equation*}
  Therefore, we arrive at 
  \begin{equation*}
    \begin{aligned}
           & {1\over2\delta t}\bigg(\|\errc ^{n+1}\|^2-\|\errc ^{n}\|^2+2c_{40}\int_{\partial\Omega}{(\partial_\bfn  \errc ^{n+1})}^2-{(\partial_\bfn  \errc ^{n})}^2\md S+\|R^{n+1}\|^2-\|R^{n}\|^2\bigg) \\
      \leq & \, \|\errc ^{n+1}\|^2+\|R^{n+1}\|^2+\|T_c^n\|^2+\|T^n\|^2+c_{40}\int_{\partial\Omega}{(T_{\partial\Omega}^n)}^2+{(\partial_\bfn  \errc ^{n+1})}^2\md S                                      \\
           & -(c_{00}-\gamma)\big(\|\errc ^{n+1}\|^2+\|\errc ^n\|^2+\|\theta_c^n\|^2\big)-2c_{02}\big(\|R^{n+1}\|^2+\|R^n\|^2+\|\Theta^n\|^2\big),
    \end{aligned}
  \end{equation*}
  that is,
  \begin{equation*}
    \begin{aligned}
      {1\over2\delta t}\big(\varepsilon^{n+1}-\varepsilon^n\big)
        \leq &\, C\bigg(\varepsilon^{n+1}+\varepsilon^{n}+\|T_c^n\|^2+\|T^n\|^2              \\
       & +\|\theta_c^n\|^2+\|\Theta^n\|^2+\int_{\partial\Omega}{(T_{\partial\Omega}^n)}^2\md S\bigg),
    \end{aligned}
  \end{equation*}
  where $ C $ depends on $(c_{00}-\gamma), c_{02}$ and $c_{40}$.
  The truncation errors can be estimated as
  \begin{equation*}
    \begin{aligned}\label{truncation}
        \|T_c^n\|^2+\|T^n\|^2&\leq C\delta t\int_{t^n}^{t^{n+1}}\|\partial_{tt}c\|^2+\|\partial_{tt}Q\|^2\md t,                                                         \\
        \|\theta_c^n\|^2+\|\Theta^n\|^2&\leq C\delta t\int_{t^n}^{t^{n+1}}\|\partial_{t}c\|^2+\|\partial_{t}Q\|^2\md t,                                                           \\
        \int_{\partial\Omega}{(T_{\partial\Omega}^n)}^2\md S&\leq C\delta t\int_{t^n}^{t^{n+1}}\int_{\partial\Omega}{(\partial_{tt}\partial_{\bfn }c)}^2\md S\,\md t.
    \end{aligned}
  \end{equation*}
  Using Gronwall's inequality, we deduce (\ref{Error}).
\end{proof}

\subsection{Full Discretization}
In the spatial discretizations, to keep the properties above, it is significant for the energy and its variation to be consistent. 
To this end, we shall discretize the energy first, then derive the scheme by taking derivatives w.r.t. the discrete variables. 
In particular, the normal derivatives on the boundary need to be handled carefully.

Here, we consider a square region $\Omega={[0,L]}^2$, and divide it into $N^2$ cells uniformly.
Denote $h=L/N$. 
We denote by the index $(l,m)$ the point $(lh,mh)$, 
and by $\Gamma_{l,m}$ the cell whose lower left index is $(l,m)$. 
In addition, for the discussions of the boundary equation, we introduce the notations $\sigma^x_{l,0}$ (resp. $\sigma^x_{l,N}$) to represent the boundary cell $x\in [lh,(l+1)h]$, $y=0$ (resp. $y=L$), and $\sigma^y_{0,m}$ (resp. $\sigma^y_{N,m}$) for the boundary cell $y\in [lh,(l+1)h]$, $x=0$ (resp. $x=L$).

In the free energy, we need to discretize $\partial _1 u\partial_1 v$,
$\partial_1u\partial_2 v$, $\partial_{11}u\partial_{11}v$ and
$\partial_{11}u\partial_{22}v$ on each cell. 
Generally speaking, discrete second order spatial derivatives are defined on the nodes.
But first order spatial derivatives are different for various cases. 

On the cell $\Gamma_{l,m} $, $\partial _1 u\partial_1 v$ is discretized as
\begin{equation}\label{D_1D_1}
  \begin{aligned}
    (D_1 uD_1v)_{\Gamma_{l,m}}=
    & {1\over 2}\left({u_{l+1,m}-u_{l,m}\over h}\right)\left({v_{l+1,m}-v_{l,m}\over h}\right)            \\
    & +{1\over 2}\left({u_{l+1,m+1}-u_{l,m+1}\over h}\right)\left({v_{l+1,m+1}-v_{l,m+1}\over h}\right).
  \end{aligned}
\end{equation}
The variational derivative of $\partial _1 u\partial_1 v$ w.r.t. $v$ is $-\partial_{11}u$. 
Its discretization at $(l,m)$ is obtained by taking the derivative w.r.t. $v_{l,m}$. Here, we need to sum up all the cells involving $v_{l,m}$ for the term $D_1uD_1v$. It leads to the commonly used second difference, 
\begin{equation}\label{D11u}
  -(D_{11} u)_{l,m}=-{u_{l-1,m}-2u_{l,m}+u_{l+1,m}\over h^2}.
\end{equation}
The term $\partial_1u\partial_2 v$ is discretized as
\begin{equation}
  {(D_1uD_2v)}_{\Gamma_{l,m}} = \frac{1}{4h^2}\left(u_{l+1,m+1}+u_{l+1,m}-u_{l,m+1}-u_{l,m}\right)\left(v_{l+1,m+1}-v_{l+1,m}+v_{l,m+1}-v_{l,m}\right). \nonumber
\end{equation}
The discretization of its variational derivative, $-\partial_{12}u$, is calculated by taking the derivative w.r.t. $v_{l,m}$, yielding 
\begin{equation}
  -(D_{12}u)_{l,m}=-(D_{21}u)_{l,m}=\frac{-u_{l-1,m-1}-u_{l+1,m+1}+u_{l-1,m+1}+u_{l+1,m-1}}{4h^2}.\nonumber
\end{equation}
Then we turn to the squared second order derivatives in the free energy. 
Since second order derivatives have been defined on the nodes, it is natural to discretize $\partial_{11} u\partial_{11}v$ as the average on the four nodes of the cell, i.e. 
\begin{equation}\label{D11uD11v}
  \begin{aligned}
    {(D_{11} uD_{11}v)}_{\Gamma_{l,m}}= 
    & {1\over4}\Big[{(D_{11} uD_{11}v)}_{l,m}+{(D_{11} uD_{11}v)}_{l+1,m} \\
      & +{(D_{11} uD_{11}v)}_{l,m+1}+{(D_{11} uD_{11}v)}_{l+1,m+1}\Big],
  \end{aligned}
\end{equation}
where $(D_{11}u)_{l,m}$ is deduced above in \eqref{D11u}.
Because $\partial_{1111}u$ is the variational derivative of $\partial_{11} u\partial_{11}v$ w.r.t. $v$, we deduce its discretization by taking the derivatives of $v_{l,m}$. 
For $2\le l,m\le N-2$, the sum of terms involving $v_{l,m}$ is given by 
\begin{equation}
  {(D_{11} uD_{11}v)}_{l-1,m}+{(D_{11} uD_{11}v)}_{l,m}+{(D_{11} uD_{11}v)}_{l+1,m}. \nonumber
\end{equation}
Thus, we deduce that $\partial_{1111}u$ is discretized as 
\begin{equation}
  (D_{1111} u)_{l,m}={u_{l-2,m}-4u_{l-1,m}+6u_{l,m}-4u_{l+1,m}+u_{l+2,m}\over h^4}.\label{D1111}
\end{equation}
The term $\partial_{11} u\partial_{22} v$ is handled in the same way, 
yielding 
\begin{equation*}
  \begin{aligned}
    (D_{1122} u)_{l,m}={1\over h^4}\big(
    & u_{l-1,m-1}-2u_{l-1,m}+u_{l-1,m+1}-2u_{l,m-1}+4u_{l,m}-2u_{l,m+1} \\
    & +u_{l+1,m-1}-2u_{l+1,m}+u_{l+1,m+1}\big).
  \end{aligned}
\end{equation*}
For $l=1,N-1$, or $m=1,N-1$, the discretized free energy involving the node $(l,m)$ are different. 
For example, the sum of terms \eqref{D11uD11v} involving $v_{1,m}$ is 
\begin{equation}\nonumber
  \frac{1}{2}{(D_{11} uD_{11}v)}_{0,m}+{(D_{11} uD_{11}v)}_{1,m}+{(D_{11} uD_{11}v)}_{2,m}. 
\end{equation}
To be consistent with the case $2\le l,m\le N-2$, we introduce ghost nodes outside and adopt the discretization of $\partial_{\bfn } v$ on the boundary nodes below,  
\begin{align}
  &(D_{\bfn }^xv)_{0,m}=\frac{v_{-1,m}-v_{1,m}}{2h},\
  (D_{\bfn }^xv)_{N,m}=\frac{v_{N+1,m}-v_{N-1,m}}{2h},\nonumber\\
  &(D_{\bfn }^yv)_{l,0}=\frac{v_{l,-1}-v_{l,1}}{2h},\
  (D_{\bfn }^yv)_{l,N}=\frac{v_{l,N+1}-v_{l,N-1}}{2h}. \label{D_n}
\end{align}
Then, we can write 
\begin{equation}\label{D11_0m}
\frac{1}{2}(D_{11}v)_{0,m}=\frac{1}{h}\big((D_{\bfn }^xv)_{0,m}+\frac{1}{h}(v_{1,m}-v_{0,m})\big). 
\end{equation}
Now, we regard these discrete normal derivatives as independent variables in addition to $v_{l,m}$ for $1\le l,m\le N-1$. 
Under this viewpoint, we take derivative w.r.t. $v_{1,m}$ to obtain \eqref{D1111}. 

The truncation errors of the discretizations above are $O(h^2)$, which are summarized below. 
\begin{lemma}\label{trunc}
For $u\in C^6(\overline{\Omega})$, we have 
\begin{align}
  |(\partial_{11}u-D_{11}u)_{l,m}|\leq{h^2\over12}\max_{\overline{\Omega}} |\partial_{1111}u|&,\ 
  |(\partial_{12}u-D_{12}u)_{l,m}|\leq\frac{h^2}{3}\max_{\overline{\Omega},k=1,2}|\partial_{12kk}u|,\nonumber\\
  |(\partial_{1111}u-D_{1111} u)_{l,m}|\leq {h^2\over6}\max_{\overline{\Omega}}|\partial_{111111}u|&,\ 
  |(\partial_{1122}u-D_{1122} u)_{l,m}|\leq {h^2\over6}\max_{\overline{\Omega},k=1,2}|\partial_{1122}\partial_{kk}u|.\nonumber
\end{align}
Let $\Omega_h=[-h,L+h]^2$.
For $v\in C^3(\overline{\Omega_h})$, we have 
\begin{align*}
  \left|(\partial_{\bfn}v-D_{\bfn}^xv)_{0,m}\right|\leq h^2\max_{\overline{\Omega_h}}|\partial_{111}v|,\quad
  \left|(\partial_{\bfn}v-D_{\bfn}^yv)_{l,0}\right|\leq h^2\max_{\overline{\Omega_h}}|\partial_{222}v|.
\end{align*}
\end{lemma}

The discretizations keep the positive definiteness of quadratic terms. 
\begin{lemma}\label{sq_dis}
  For any $u$ and $v$,
  \begin{equation*}
    h^2{( D_1u D_1u + D_2v D_2v + 2D_1u D_2v)}_{\Gamma_{l,m}}\ge 0. 
  \end{equation*}
\end{lemma}
\begin{proof}
  It can be calculated as 
  \begin{align*}
    & \frac{1}{4}\Big({\left(u_{l+1,m} - u_{l,m}+v_{l,m+1} - v_{l,m}\right)}^2+{\left(u_{l+1,m} - u_{l,m}+v_{l+1,m+1} - v_{l+1,m}\right)}^2\\
    &\quad +{\left(u_{l+1,m+1} - u_{l,m+1}+v_{l,m+1} - v_{l,m}\right)}^2+{\left(u_{l+1,m+1} - u_{l,m+1}+v_{l+1,m+1} - v_{l+1,m}\right)}^2\Big). 
  \end{align*}
\end{proof}

\begin{corollary}\label{dis_spd}
  When $4\gamma_1c_{22}\ge c_{23}^2$, we have
  $$
  (c_{22}D_iQ_{ij}D_kQ_{kj} + c_{23}D_iQ_{ij}D_jc + \gamma_1D_jcD_jc)_{\Gamma_{l,m}}\ge 0.
  $$
\end{corollary}
\begin{proof}
  Write it as 
  \begin{align*}
  &c_{22}D_1(Q_{11}+c)D_1(Q_{11}+c)+c_{23}D_1(Q_{11}+c)D_2Q_{21}+\gamma_1D_2Q_{21}D_2Q_{21}\\
    &+c_{22}D_1Q_{11}D_1Q_{11}+c_{23}D_1Q_{11}D_2(Q_{21}+c)+\gamma_1D_2(Q_{21}+c)D_2(Q_{21}+c). 
  \end{align*}
  Then use Lemma \ref{sq_dis} for each line. 
\end{proof}

For $u$ defined on interior nodes, denote its summation over all interior nodes as $\sum_P u_P$.
For boundary nodes, they belong to some boundary cells.
When specifying a boundary cell $\sigma$, the normal derivative has a definite meaning of whether it is $D^x_{\bfn}$ or $D^y_{\bfn}$, so we may denote it as $D_{\bfn}$.
Then, we denote by $u_{\sigma}$ the average of $u$ on the two endpoints of a cell $\sigma$, and by $\sum_{\sigma} u_{\sigma}$ the summation of $u$ over all boundary cells. 
For $v$ defined on interior cells, denote its summation over all cells as $\sum_\Gamma v_\Gamma$.

Let $\Delta_h,\,\Delta_h^2$ denote $D_{11}+D_{22}$ and $D_{1111}+2D_{1122}+D_{2222}$, respectively.
The summation by parts below can be established. 
\begin{lemma}\label{sumparts_first}
If $u$ takes zero on all boundary nodes $(l=0,N\,\mathrm{or}\,m=0,N)$, then
we have
\begin{align*}
    & -\sum_P{(uD_{11}v)}_P=\sum_\Gamma{(D_1uD_1v)}_\Gamma,                                 \nonumber\\
    & -\sum_P{(uD_{12}v)}_P=\sum_\Gamma{(D_1uD_2v)}_\Gamma=\sum_\Gamma{(D_2uD_1v)}_\Gamma.
\end{align*}
\end{lemma}
\begin{lemma}\label{sumparts_bnd}
  For any $u, v$, it holds
  \begin{equation}\label{Dkc}
    h^2\sum_{\Gamma}{(D_k u D_k v)}_{\Gamma}=h\sum_{\sigma} {(uD_{\bfn }v)}_{\sigma}-h^2\sum_\Gamma{(u\Delta_h v)}_\Gamma,
  \end{equation}
  If $v$ satisfies $v|_{\partial\Omega}=0$, it also holds
  \begin{equation}\label{Dlc}
    h^2\sum_{\Gamma}{(\Delta_h u\Delta_h v)}_{\Gamma}=h\sum_{\sigma} {(\Delta_h uD_{\bfn }v)}_{\sigma}+h^2\sum_P{(v\Delta^2_h u)}_P.
  \end{equation}
\end{lemma}
The proof is left to Appendix.

Now we are able to write down a splitting similar to \eqref{splitting}.
Define 
\begin{equation}\nonumber
  \hat{E}(c,Q)=h^2\sum_P{\big(\qeN(c,Q)\big)}_P+\hat{E}_+(c,Q)-\hat{E}_-(c,Q)+A,
\end{equation}
with 
\begin{align*}
    \hat{E}_+(c,Q)= & h^2\sum_\Gamma \Big(c_{20}D_icD_ic+c_{21}D_iQ_{jk}D_iQ_{jk}+c_{22}D_iQ_{ij}D_kQ_{kj}+c_{23}D_jcD_iQ_{ij}\\
    &\quad +c_{40}(\Delta_h c)^2\Big)_\Gamma+h^2\sum_P{\big(\gamma c^2\big)}_P,  \\
    \hat{E}_-(c,Q)= & h^2\sum_P{\big(-(c_{00}-\gamma)c^2-c_{02}Q^2\big)}_P,
\end{align*}
and $A$ is a constant determined by the value of $ c $ and $ Q $ on the boundary nodes.

\begin{proposition}\label{dis_cvx}
  $\hat{E}_+(c,Q)$ is convex w.r.t. $c_{l,m},Q_{l,m},(D_{\bfn}c)_{l,m}$. 
\end{proposition}
\begin{proof}
  Follow the same procedure as Proposition \ref{gamma_convex}, where Corollary \ref{dis_spd} and the summation by parts \eqref{Dkc} is necessary. 
\end{proof}

Define
\begin{equation}\nonumber
  \begin{aligned}
    L_c[c,Q]        & =2\gamma c-2c_{20}D_{kk}c-c_{23}D_{ij}Q_{ij}+2c_{40}\Delta_h^2c,  \\
    L_{Q_{ij}}[c,Q] & =-2c_{21}D_{kk}Q_{ij}-2c_{22}D_{ik}Q_{kj}-c_{23}D_{ij}c.
  \end{aligned}
\end{equation}
The definition is given on each node.
Hereafter we omit the index $(l,m)$ when it is unnecessary to specify. 
The above discretizations keep the structures of the energy variation given by \eqref{delta_E}. 
\begin{proposition}\label{dis_diff}
The differential of the energy is given by 
\begin{equation}\nonumber
  \begin{aligned}
    \md \hat{E}= & h^2\sum_P{\bigg(\Big(\frac{\partial \qeN(c,Q)}{\partial c}+2(c_{00}-\gamma)c+L_c[c,Q]\Big)\md c\bigg)}_P+2c_{40}h\sum_{\sigma}{(\Delta_h c \md D_{\bfn } c)}_{\sigma} \\
    & +h^2\sum_P{\bigg(\Big(\frac{\partial \qeN(c,Q)}{\partial Q}+2c_{02}Q_{ij}+L_{Q_{ij}}[c,Q]\Big)\md Q_{ij}\bigg)}_P.
  \end{aligned}
\end{equation}
\end{proposition}
\begin{proof}
  For interior points, the differential is given by our construction. 
  For the boundary normal derivatives, only the $\Delta_hc\Delta_hc$ term is involved.
  Using the summation by parts \eqref{Dlc} and the fact that $\md c=0$ on boundary nodes, we have
  \begin{equation}
    \nonumber
    h^2\md\sum_{\Gamma}{(\Delta_h c\Delta_h c)}_{\Gamma}=2h\sum_{\sigma} {(\Delta_h c\md D_{\bfn }c)}_{\sigma}+2h^2\sum_P{(\Delta^2_h c\md c)}_P.
  \end{equation}
\end{proof}

Now we are ready to write down the fully discretized scheme, 
\begin{align}
  & \frac{c^{n+1}-c^{n}}{\delta t}=-\hat{\scrG }\Big(\frac{\partial \qeN(c^{n+1},Q^{n+1})}{\partial c^{n+1}}+2(c_{00}-\gamma)c^{n}+L_c[c^{n+1},Q^{n+1}]\Big),  \label{discrete_1}\\
  & \frac{Q_{ij}^{n+1}-Q_{ij}^{n}}{\delta t}=-\scrP \Bigg({\left(\frac{\partial \qeN(c^{n+1},Q^{n+1})}{\partial Q^{n+1}}\right)}_{ij}+2c_{02}Q_{ij}^n+L_{Q_{ij}}[c^{n+1},Q^{n+1}]\Bigg), \label{discrete_2}\\
  & \frac{D_{\bfn }c^{n+1}-D_{\bfn }c^{n}}{\delta t}=-\Delta_h c^{n+1},\quad\mathrm{on\, \partial\Omega},  \label{discrete_3}
\end{align}
where the projection operator $\scrG $ is also discretized as $\hat{\scrG }$, satisfying
\begin{equation*}
  \hat{\scrG }A=A-{1\over \# P}\sum_{P}A_P,
\end{equation*}
where $\# P$ is the number of the interior nodes.
By the midpoint rule of integral, $\scrG$ also yields a truncation error of $O(h^2)$.

Since the above full discretized scheme is consistent with the time discretized
scheme \eqref{t_discrete_c}--\eqref{t_discrete_bnd}, they satisfy the same properties and the proofs are
also similar. In the following, we only briefly discuss the energy dissipation
and error estimation.

\begin{theorem}
  For any $\delta t$ and $h$, the scheme \eqref{discrete_1}--\eqref{discrete_3} has a unique solution. Moreover, the scheme satisfies the discrete energy law,
  \begin{equation}\nonumber 
    \begin{aligned}
       & \hat{E}[c^{n+1},Q^{n+1}]+h^2\Big[\frac{1-(c_{00}-\gamma)\delta t}{\delta t}\sum_{P}{(c^{n+1}-c^n)}_P^2+\frac{1-c_{02}\delta t}{\delta t}\sum_{P}{(Q_{ij}^{n+1}-Q_{ij}^n)}_P^2\Big] \\
       & +\hat{E}_+(c^{n+1}-c^n,Q^{n+1}-Q^n)+2c_{40}\frac{h}{\delta t}\sum_{\sigma}{(D_{\bfn}c^{n+1}-D_{\bfn}c^n)}_{\sigma}^2\leq \hat{E}[c^{n},Q^{n}].
    \end{aligned}
  \end{equation}
\end{theorem}
\begin{proof}
  The proof follows a similar procedure as Theorem~\ref{energylaw}.
  Consider the function of $c^{n+1}_{l,m}$, $Q^{n+1}_{l,m}$, $(D_{\bfn}c)_{l,m}$, 
  \begin{equation}\nonumber
    \begin{aligned}
      \hat{F}[c^{n+1}, & Q^{n+1}]=\frac{1}{2\delta t}\bigg(h^2\sum_{P}\Big({\left(c^{n+1}-c^n\right)}^2_P+{\left(Q_{ij}^{n+1}-Q_{ij}^n\right)}^2_P\Big) +2c_{40}h\sum_{\sigma}(D_{\bfn}c^{n+1}-D_{\bfn}c^n)_{\sigma}^2\bigg)\\
                       & +h^2\left(2(c_{00}-\gamma){\left(c^{n}\cdot c^{n+1}\right)}_P+2c_{02}{\left(Q_{ij}^n\cdot Q_{ij}^{n+1}\right)}_P+\sum_P{\big(\qeN(c^{n+1},Q^{n+1})\big)}_P\right)+\hat{E}_+\left(c^{n+1},Q^{n+1}\right),
  \end{aligned}\end{equation}
where $c^{n+1}_{l,m}$ satisfies
\begin{equation*}
\sum_{P} c^{n+1}_P = \sum_{P} c^n_P.
\end{equation*}
  To establish that it is strictly convex and lower bounded, we need Corollary \ref{dis_spd} and Proposition \ref{dis_cvx}.
  Therefore it possesses a unique minimizer such that $(c^{n+1},Q^{n+1})_{l,m}\in \overline{\cQdomM}$. Then, we are able to verify that the minimizer must be taken with $(c^{n+1},Q^{n+1})_{l,m}\in (\cQdomM)^{\circ}$ by a similar approach in Theorem \ref{energylaw}.
  It enables us to take derivatives of discrete variables to arrive at 
  \begin{equation*}
    \hat{\scrG } \Big(\frac{\partial \hat{F}}{\partial {(c^{n+1})}_{l,m}}\Big)=0,\quad\scrP \Big(\frac{\partial \hat{F}}{\partial {(Q_{ij}^{n+1})}_{l,m}}\Big)=0,\quad \frac{\partial \hat{F}}{\partial {(D_{\bfn}c^{n+1})_{l,m}}}=0. 
  \end{equation*}
  which is exactly \eqref{discrete_1}--\eqref{discrete_3} using Proposition \ref{dis_diff}.

  The energy law can be established by taking the dot product of \eqref{discrete_1}--\eqref{discrete_3} with $c^{n+1}-c^n$, $Q_{ij}^{n+1}-Q_{ij}^n$, $D_{\bfn }c^{n+1}-D_{\bfn }c^{n}$, respectively, followed by taking the $\sum_P$ and $\sum_{\sigma}$.
  The summation by parts, Lemma \ref{sumparts_first} and \ref{sumparts_bnd} are needed. 
\end{proof}

Define the error for the full discretization as
\begin{equation}\label{epsilon}
  \varepsilon^n\triangleq h^2\sum_P\Big(|\errc ^n|_P^2+|R^n|_P^2\Big)+h\sum_{\sigma}{\big(D_{\bfn }\errc ^n\big)}_{\sigma}^2.
\end{equation}
\begin{theorem}
  Assume that $c(\mr,t)\in C^6(\Omega\times[0,T]), Q(\mr,t)\in C^4(\Omega\times[0,T])$. The error of the scheme \eqref{discrete_1}--\eqref{discrete_3} has the estimate
  \begin{equation}\nonumber 
    \varepsilon^n\leq C(\delta t^2+h^4).
  \end{equation}
  The constant $ C $ depends on $T$, the coefficients $c_{02},c_{40}$, and the maximum derivatives of $c,Q_{ij}$ in $\Omega\times[0,T]$.
\end{theorem}
\begin{proof}
  The equations for $\errc ^n,R_{ij}^n$ are derived as
  \begin{align*}
    & \frac{\errc ^{n+1}-\errc ^{n}}{\delta t}=-\hat{\scrG} \Big(\frac{\partial \qeN(c^{n+1}(\mr),Q^{n+1}(\mr))}{\partial c^{n+1}(\mr)}-\frac{\partial \qeN(c(\mr,t^{n+1}),Q(\mr,t^{n+1}))}{\partial c(\mr,t^{n+1})}                 \\
    & \qquad+2(c_{00}-\gamma)\errc ^{n}+L_c[\errc ^{n+1},R^{n+1}]\Big)+U_h^n. \\
    & \frac{R_{ij}^{n+1}-R_{ij}^{n}}{\delta t}=-\scrP \Big(\frac{\partial \qeN(c^{n+1}(\mr),Q^{n+1}(\mr))}{\partial Q_{ij}^{n+1}(\mr)}-\frac{\partial \qeN(c(\mr,t^{n+1}),Q(\mr,t^{n+1}))}{\partial Q_{ij}(\mr,t^{n+1})} \\
    & \qquad+2c_{02}R_{ij}^n+L_{Q_{ij}}[\errc ^{n+1},R^{n+1}]\Big)+V_h^n.  \\
    &{D_{\bfn }\errc ^{n+1}-D_{\bfn }\errc ^{n}\over\delta t}=-\Delta_h \errc ^{n+1}+Z_{B}^n. 
  \end{align*}
  The truncation errors can be estimated using Lemma \ref{trunc} and the results for the time discretization in (\ref{truncation}), 
  \begin{equation}\nonumber
    |U_h^n|+|V_h^n|\leq C(\delta t+h^2), \qquad |Z_{B}^n|\leq C(\delta t+h^2),
  \end{equation}
  with $ C $ dependent on the maximum derivatives on $\Omega\times[0,T]$. Take dot product on three equations with $\errc ^{n+1}$, $R_{ij}^{n+1}$, $D_{\bfn}\errc^{n+1}$, respectively.
  Armed with summation by parts, we are able to follow the same derivation of (\ref{Error}) to conclude the proof.
\end{proof}

\section{Numerical Results}\label{NE}
Computations are carried out in the square domain $\Omega=[0,L]\times[0,L]$ where $L$ takes a few different values. 
Fix $M=100,\nu=0.1$, while the average concentration $c_0$ varies (defined in \eqref{avc}). 
The parameters in the free energy are chosen according to the calculations in \cite{Han2015From} from the hard repulsive potential, which prove to effectively describe the smectic structures \cite{Mei2015On}: 
$c_{00}=0.112$, $c_{02}=-0.0735$, $c_{20}=-0.00442$, $c_{21}=0.00091$, $c_{22}=0.0034$, $c_{23}=-0.0083$, and $c_{40}=0.00014$. 
The scheme \eqref{t_discrete_c}--\eqref{t_discrete_bnd} is nonlinear, which we solve by damped Newton's method: whenever the iteration goes out of the domain $(\cQdomM)^{\circ}$, the step size is halved until the trial lies within $(\cQdomM)^{\circ}$. 

To illustrate the numerical results, we use $\lambda_\mathrm{max}(Q/c)$ to denote the maximum eigenvalue of $Q/c$, which quantifies the intensity of anisotropy. 
It ranges in $[0,0.5)$, with the value $\lambda_\mathrm{max}(Q/c)=0$ representing the isotropic state. 
The corresponding eigenvector indicates the direction along which the molecules are aligning. 

\subsection{Uniform boundary concentration anchorings}
\begin{figure}[htbp]
  \centering
  \includegraphics[width=1.0\columnwidth]{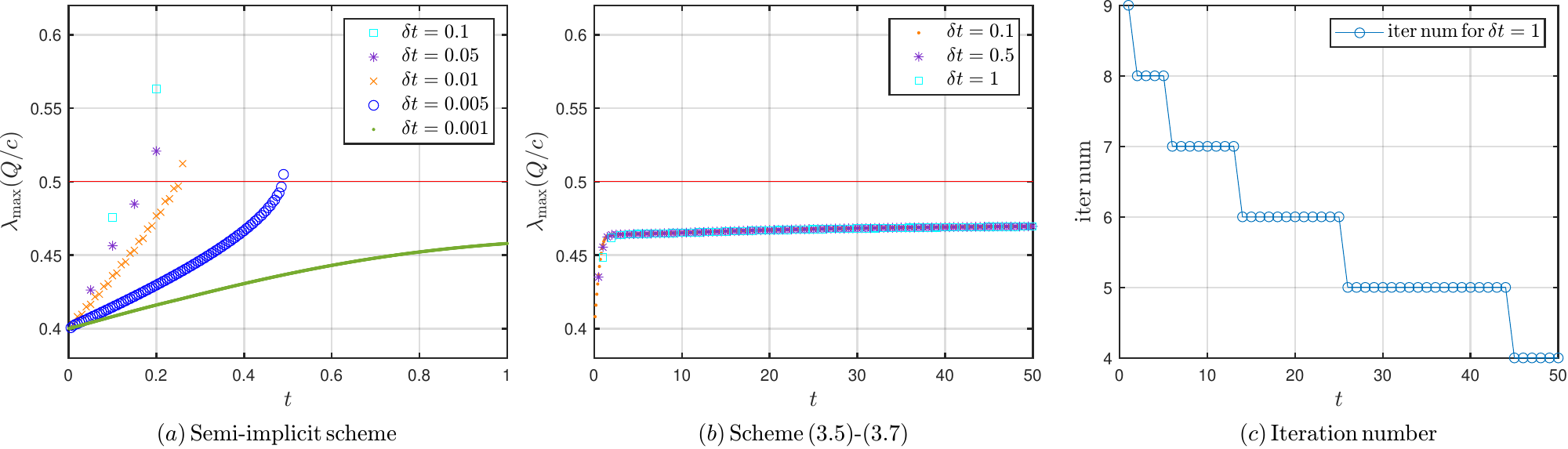}
  \caption{ (a) Maximum eigenvalues of $Q/c$ for the semi-implicit scheme. (b) Maximum eigenvalues of $Q/c$ for the scheme \eqref{t_discrete_c}--\eqref{t_discrete_bnd}. (c) Number of Newton's iteration for $\delta t=1$.}\label{perf}
\end{figure}

On the boundary $\partial\Omega$, let $c=c_0, \lambda_\mathrm{max}(Q/c)=0.4$ with
the principal eigenvector chosen parallel to the boundary. 
The initial condition is chosen as $c=c_0, \lambda_\mathrm{max}(Q/c)=0.4$ with
the principal eigenvector taking a constant vector
${(-1/\sqrt{2},1/\sqrt{2})}^T$.
This setting is analogous to the confined nematic liquid crystals with parallel boundary anchorings \cite{Yin2020Construction,Han2021Solution}, but concentration variations are incorporated. 
We shall examine the results from two aspects: the performance of the scheme, and how the configurations change with $c_0$, $L$. 

\textbf{Performance of the scheme.}
We choose $L=\sqrt{10}$ and $c_0=20$. Under this average concentration, the free energy shall exhibit the smectic phase as its minimizer. However, the anchoring $c=c_0$ on the boundary would reconstruct the structures to a large extent.
To investigate the performance of the scheme, we also employ a simple semi-implicit scheme, where the linear terms are discretized implicitly while the nonlinear terms are treated explicitly.
The semi-implicit scheme is written as 
\begin{equation*}
  \begin{aligned}
    & \frac{c^{n+1}-c^n}{\delta t}=- \scrG \bigg( \frac{\partial \qeN(c^{n},Q^{n})}{\partial c^{n}}+2c_{00}c^{n+1}-2c_{20}\Delta c^{n+1}                                  \\
    & \qquad\qquad  -c_{23}\partial_{ij}Q_{ij}^{n+1}+2c_{40}\Delta^2c^{n+1}\bigg),  \\
    & \frac{Q_{ij}^{n+1}-Q_{ij}^n}{\delta t}=-  \scrP \Bigg( {\left(\frac{\partial \qeN(c^{n},Q^{n})}{\partial Q^{n}}\right)}_{ij}+2c_{02}Q^{n+1}_{ij}-2c_{21}\Delta Q_{ij}^{n+1} \\
    & \qquad\qquad-2c_{22}\partial_{ik}Q_{kj}^{n+1}-c_{23}\partial_{ij}c^{n+1}\Bigg),
  \end{aligned}
\end{equation*}
with the boundary equation conditions discretized as \eqref{t_discrete_bnd}.

As it turns out, the scheme \eqref{t_discrete_c}--\eqref{t_discrete_bnd} significantly outperforms the semi-implicit scheme.
The major point affecting the performance lies in the fact that the semi-implicit scheme is unable to ensure the numerical solution lying within $(\cQdomM)^{\circ}$. 
We plot the maximum eigenvalue of $Q/c$ in Fig. \ref{perf} (a) using various time steps.
For $\delta t$ no less than $0.005$, the maximum eigenvalue exceeds $1/2$ before $t=1$, which results in a forced stop in simulation. 
Only when $\delta t$ takes $0.001$ could the simulation be carried out through $t=1$.
In contrast, for the scheme \eqref{t_discrete_c}--\eqref{t_discrete_bnd} the computations can be done up to $\delta t=1$ (Fig. \ref{perf} (b)). 
In other words, compared with the semi-implicit scheme, a time step at least $200$ times larger can be adopted for the scheme \eqref{t_discrete_c}--\eqref{t_discrete_bnd}.
Indeed, solving the nonlinear scheme costs more in each time step due to nonlinear iterations.
Nevertheless, the number of Newton's iterations, plotted in Fig. \ref{perf} (c), is no greater than ten and decreases to four with time. 
Therefore, the increase in the computational cost for each time step can be sufficiently compensated by the capacity of using large $\delta t$.
Actually, in this example, the solution only exhibits mild variations both in time and in space, as we can realize from the energy dissipation (Fig.~\ref{cc_20_lambda2_10}(a)) and the configuration at $t=50$ (Fig.~\ref{cc_20_lambda2_10}(b)).
In this sense, constraining the numerical solution within $(\cQdomM)^{\circ}$ is a quite difficult task and deserves special care, for which an implicit discretization of the entropy term $\qeN$ is well suited.

\begin{figure}[htbp]
  \centering
  \includegraphics[width=1\columnwidth]{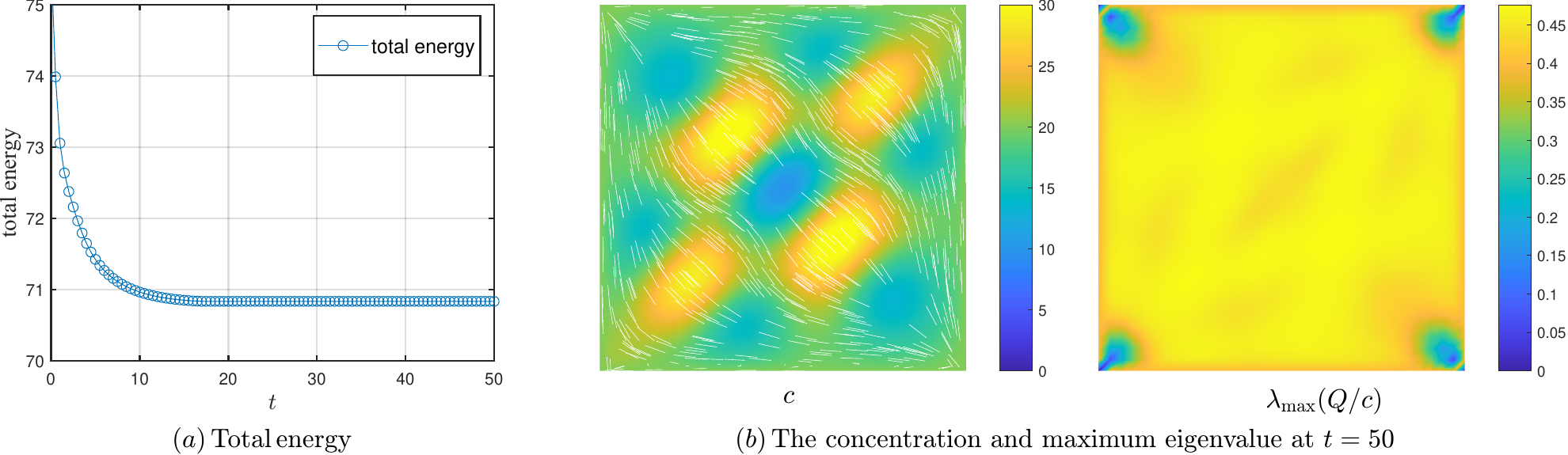}
  \caption{(a) Energy. (b) The concentration and eigenvalue at $t=50$, computed on a mesh $N=64$ and time step $\delta t=3.125\times10^{-4}$. The principal eigenvectors are plotted in the figure of concentration. }\label{cc_20_lambda2_10}
\end{figure}

Next, we examine the accuracy using the solution at $t=50$. 
The reference solution is acquired with $N=64$ and a small time step $\delta t=3.125\times10^{-4}$. 
For the temporal error, we fix the spatial resolution at $N=64$ and vary the time step as $\delta t=0.04,0.02,0.01,0.005$.
The numerical error $\varepsilon $ is plotted in Fig.\ref{error-plot}(a), confirming the first order accuracy. 
Then, for the spatial error, we choose $\delta t=h^2/10$ and vary $N=L/h$.
The error $\varepsilon $ is plotted in Fig.\ref{error-plot}(b), where a second order accuracy is recognized.

To visualize the development towards the steady state,
we present some snapshots of the solution in Fig.\ref{evolution}.
From the spatially homogeneous orientation, streamlines are formed connecting the upper-left and lower-right corners.
Then, spatial variations of the concentration gradually emerge, which disturb the streamlines to align with concentration gradients. 
Layer structures are formed along the diagonal of the square. 

\begin{figure}[htbp]
  \centering
  \includegraphics[width=0.95\columnwidth]{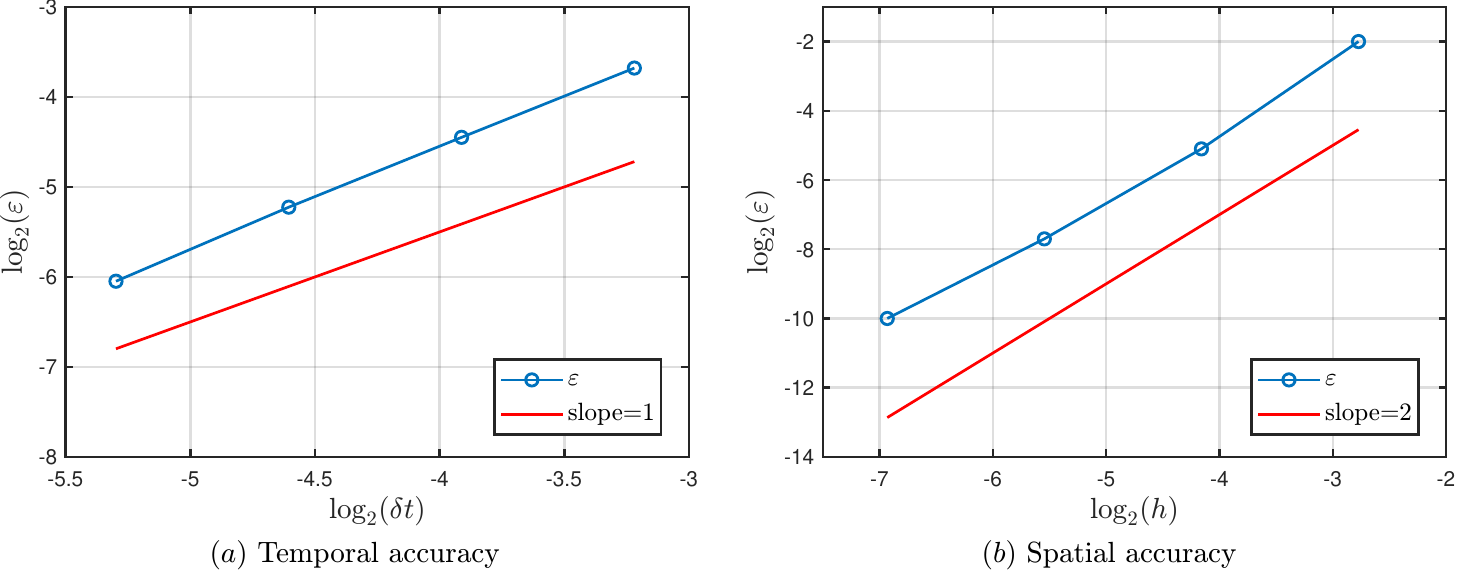}
  \caption{The error $\varepsilon $ with different (a) time steps for $N=64$; and (b) spatial discretizations with $\delta t=h^2/10$. }\label{error-plot}
\end{figure}

\begin{figure}[h]
  \centering
  \includegraphics[width=0.95\columnwidth]{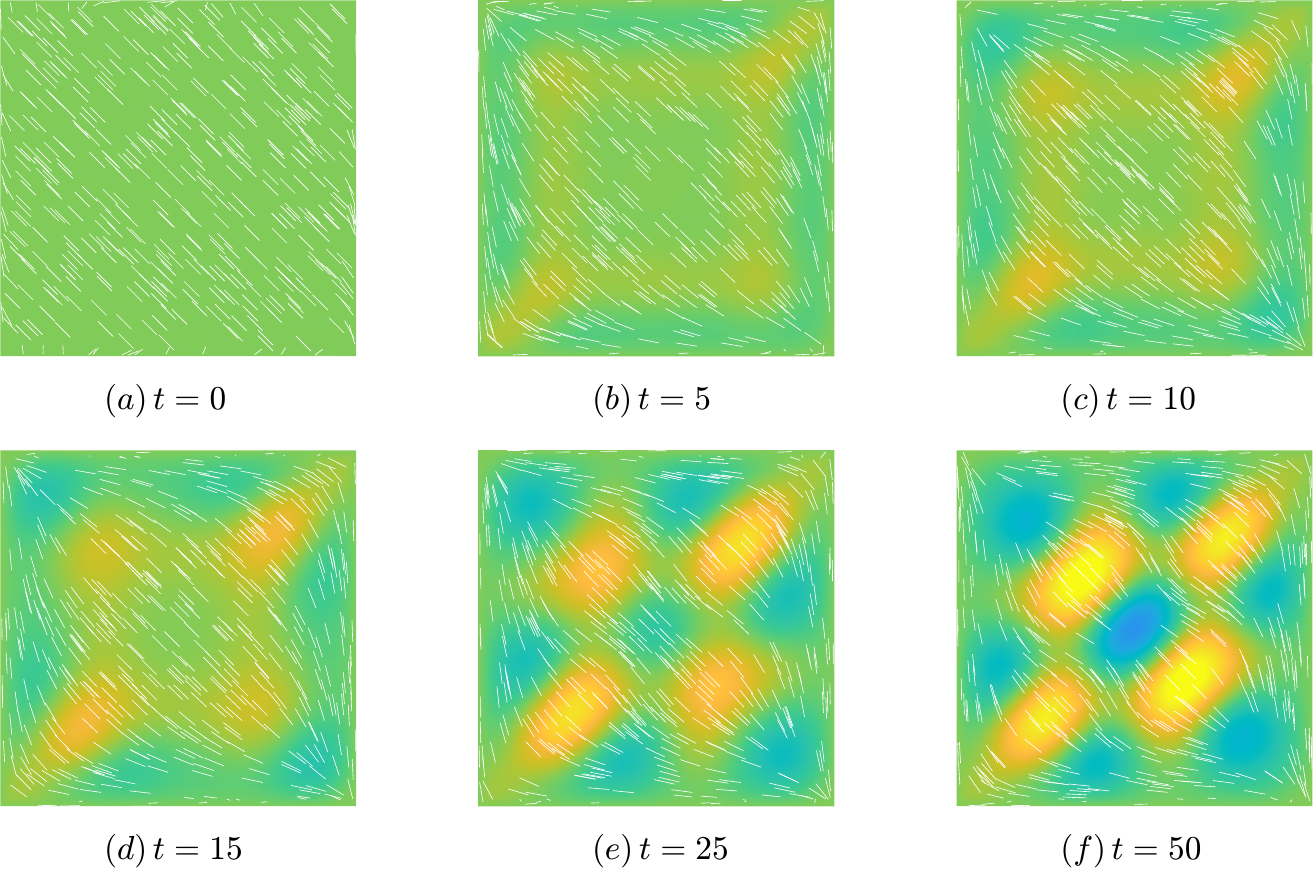}
  \caption{Snapshots of the solution for $c_0=20$ and $L=\sqrt{20}$. The colorbar is the same as the density profile shown in Fig.\ref{cc_20_lambda2_10}.}
  \label{evolution}
\end{figure}

\begin{figure}[h]
  \centering
  \includegraphics[width=0.7\columnwidth]{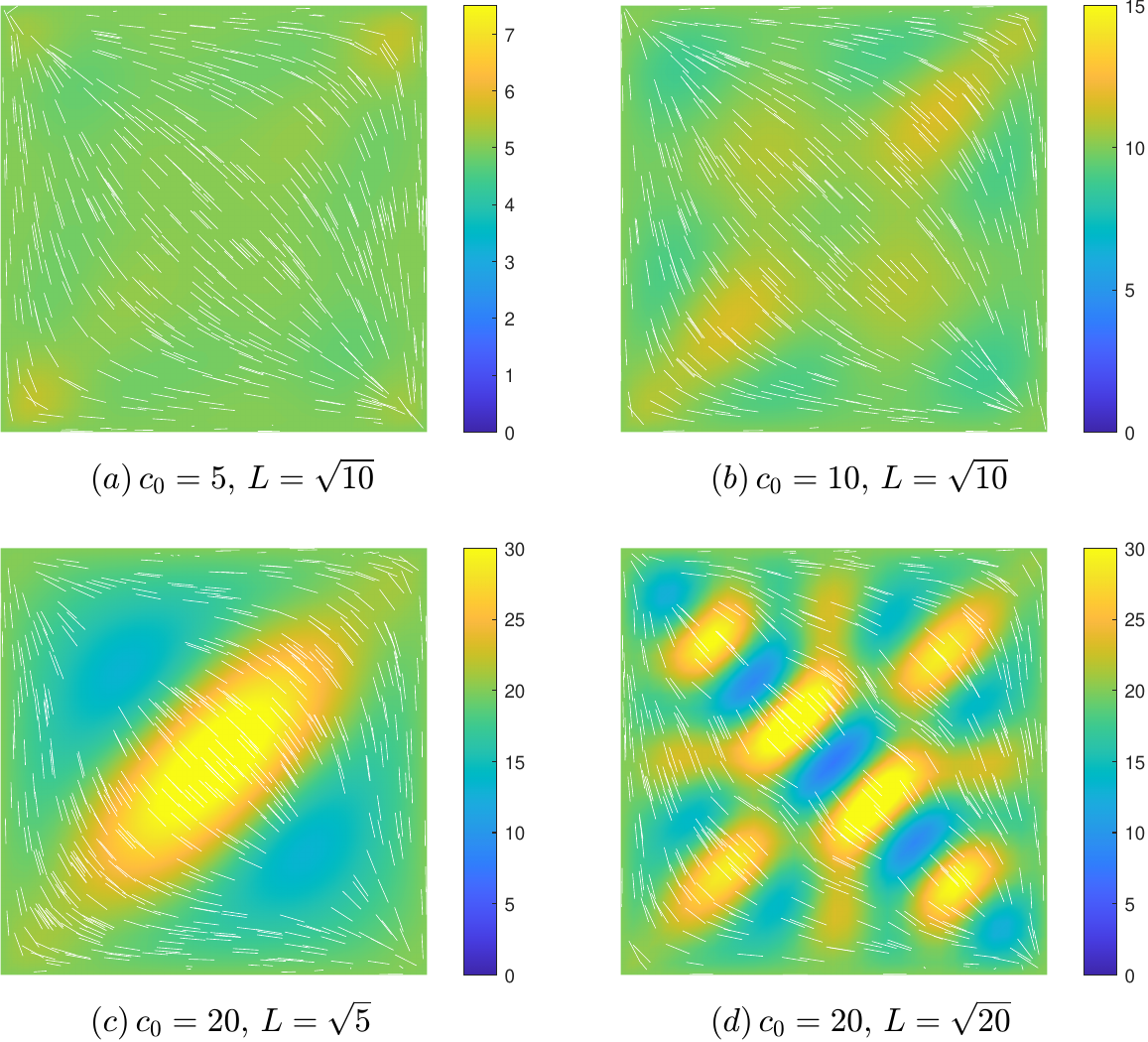}
  \caption{The steady states for (a) $c_0=5$, $L=\sqrt{10}$; (b) $c_0=10$, $L=\sqrt{10}$; (c) $c_0=20$, $L=\sqrt{5}$; (d) $c_0=20$, $L=\sqrt{20}$.}
  \label{cc_5_10}
\end{figure}

\textbf{Effect of average concentration and domain size.}
It is known that the average concentration $c_0$ greatly affects the structures. 
To comprehend its role, we examine two cases of lower concentration, $c_0=5,\,10$.
The steady states are illustrated in Fig.~\ref{cc_5_10} to compare with the configuration in Fig. \ref{cc_20_lambda2_10}.
For $c_0=5$, the result is very close to that for the nematic liquid crystals (cf. \cite{Yao2018,Yin2020Construction}).
When the average concentration increases to $c_0=10$, the layer structure can be observed ambiguously. 
The size of the computational region is another significant parameter that influences the structures.
We investigate the cases $L=\sqrt{5},\,\sqrt{20}$, drawn in Fig.~\ref{cc_5_10}.
For a smaller region $L=\sqrt{5}$, we only observe an accumulation in concentration, while for a larger region $L=\sqrt{20}$ four layers can be clearly distinguished. 

\subsection{Smectic boundary anchorings}

\begin{figure}[htpb]
  \centering
  \includegraphics[width=0.7\columnwidth]{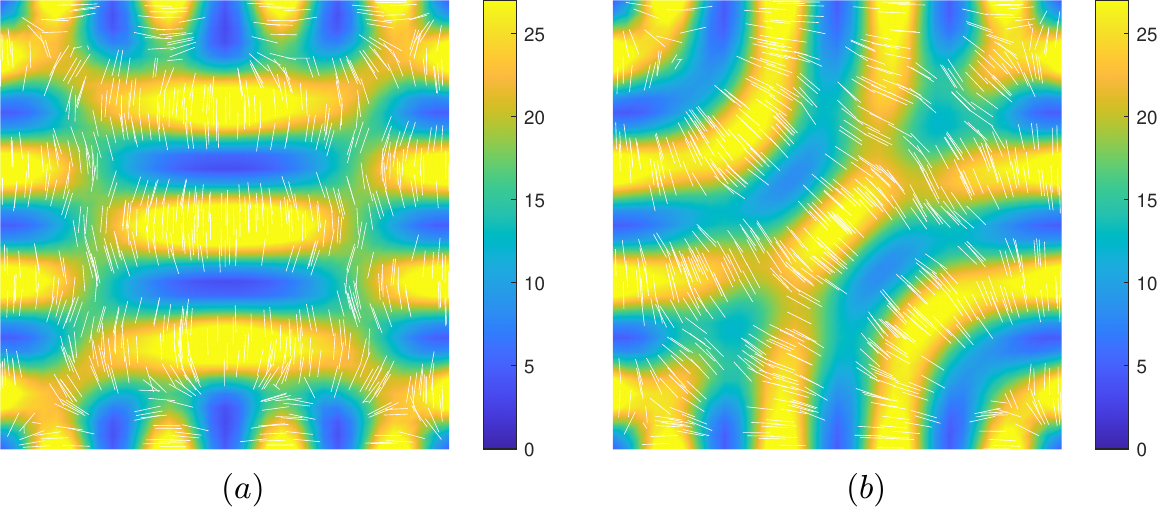}
  \caption{The steady states for $L=4d$ correspond to the initial states of (a) uniform vertical alignment and (b) uniform $135^\circ $ alignment.  }\label{alpha_4}
\end{figure}

\begin{figure}[htpb]
  \centering
  \includegraphics[width=0.9\columnwidth]{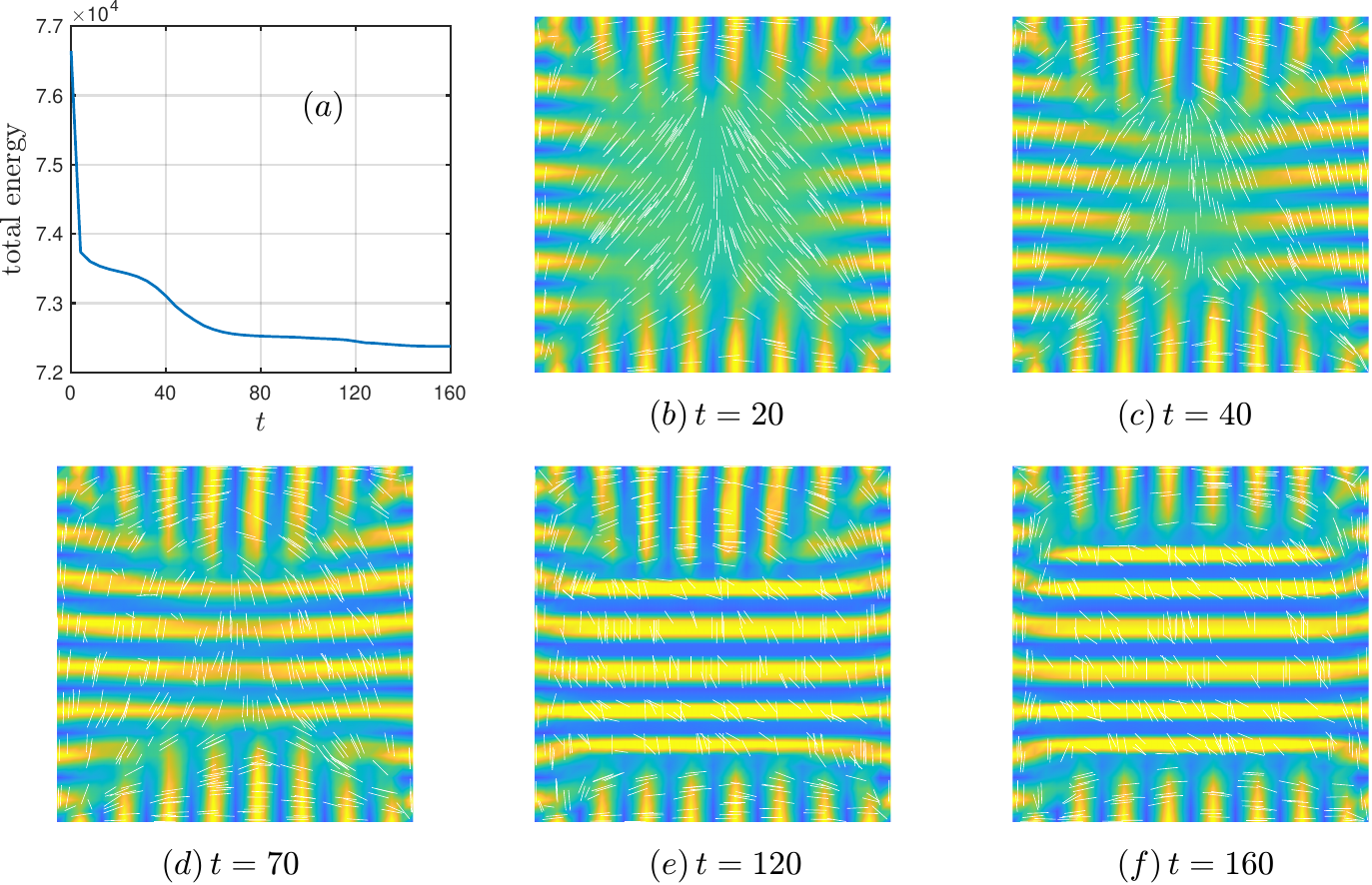}
  \caption{When the initial state consists of left-half $45^\circ $ alignment and right-half $135^\circ $ alignment for $L=8d$, the total energy versus time is shown in (a), and (b)-(f) are snapshots of the solution evolution at $t=20,40,70,120,160$.  The color bar corresponds to the density profile shown in Fig.\ref{alpha_4}. }\label{evolution2}
\end{figure}

\begin{figure}[htpb]
  \centering
  \includegraphics[width=0.9\columnwidth]{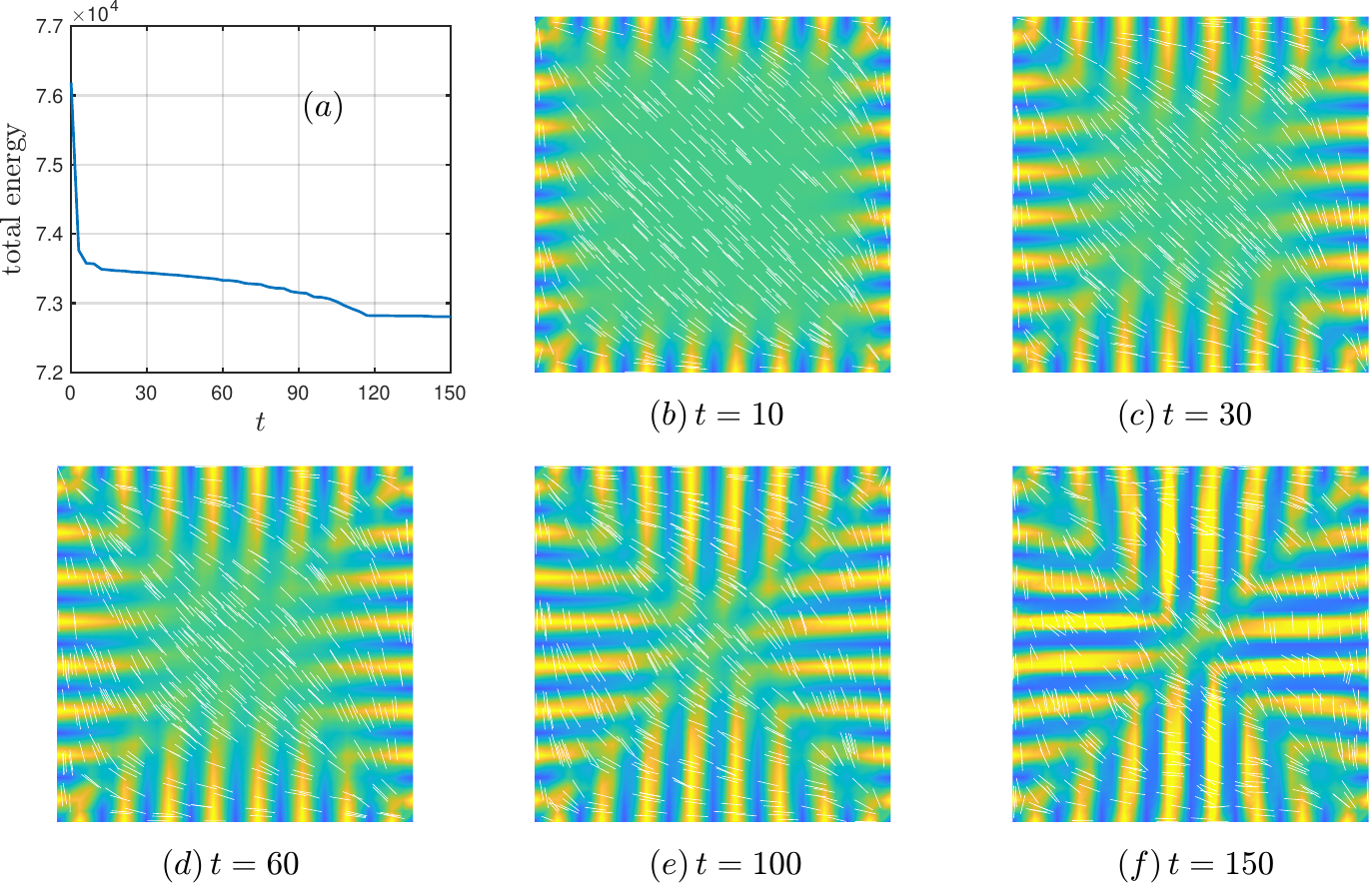}
  \caption{When the initial state is set to be $135^\circ $ alignment for $L=8d$, the total energy versus time is shown in (a), and (b)-(f) are snapshots of the solution evolution at $t=10,30,60,100,150$.  The colorbar follows the same scale as the density profile shown in Fig.\ref{alpha_4}.  }\label{evolution3}
\end{figure}

We turn to studying the configurations where the boundary is anchored with the smectic layers.
In what follows, the average density is fixed at $c_0=20$. 
To impose the boundary conditions, we need the profile of the ideal smectic phase.
It is obtained by minimizing the free energy for 1D periodic functions.
The minimization also needs to be done w.r.t. the length of period $d$.
As a result, the minimization problem is posed as follows, 
\begin{equation}\label{minimization}
  \min_{c(x),Q(x),d} {E[c(x),Q(x)]\over d}, \qquad \mathrm{s.t.}\qquad  {1\over d}\int_0^d c(x)\md x=c_0,\ c(d)=c(0),\ Q(d)=Q(0). 
\end{equation}
The optimized length of period $d$ is the layer thickness of the ideal smectic phase. 

We choose $L$ to be an integer multiple of $d$, and set the boundary values of $c$ and $Q$ according to the 1D periodic minimizer. 
Two cases $L=4d,\,8d$ are studied, with different initial conditions leading to different steady states. 
For $L=4d$, an initial state of uniform vertical alignment results in smectic layers along the side direction of the square region (Fig.\ref{alpha_4} (a)), while an initial state of uniform $135^\circ $ alignment evolves towards layers along the diagonal with bending.
For both cases, dislocations are found prone to have sharp corners. 
This pattern is more evident in the case $L=8d$, for which we also examine two initial configurations.
One has the principal vector pointing towards $45^\circ $ in the left half and $135^\circ $ in the right half of the square region.
The other has the uniform $135^\circ $ alignment.
The evolutions from the above two initial states are given in Fig. \ref{evolution2} and Fig. \ref{evolution3}, respectively. 
It is clear that the layer structures tend to keep the side directions of the square region, which are consistent with the boundary anchorings.
In particular, when two layer structures with different directions touch, their connections do not tend to be chevron or Omega-shaped that are frequently observed in layer structures without local anisotropy \cite{Ball2015Discontinuous,Xia2021Structural}, but prefer T-junction or sharp corner. 
We can also see from the snapshots that the layer structures are induced by the principal eigenvector to a large extent.
This explains the fact that in Fig. \ref{evolution2} the system eventually forms layers throughout the two opposite sides of the square region, because near the center of the region the principal vector is almost vertical.
Instead, in Fig. \ref{evolution3} the principal vector near the center of the region is roughly $135^{\circ}$, so that the layers change their directions here. 
From the energy plot we are aware of the fact that the state in Fig. \ref{evolution2} is lower than Fig. \ref{evolution3}, but the latter is indeed a minimizer since we perturb it and it evolves towards the same state.

The examples we present above indicate that in smectic phases, the interplay between the local anisotropy, the layer structures, and the defects would be quite complicated.
Multiple minimizers, i.e. locally stable states, would be commonly encountered.
The gradient flow as well as the numerical methods in this paper shall be useful for further explorations. 

\section{Conclusions}\label{conclusion}

We construct a tensor gradient flow for the smectic rod-like liquid crystals with the orientation restrained within the plane. 
The free energy is featured by the quasi-entropy as a strictly convex, lower semicontinuous function of the concentration and the tensor giving the coupled constraints. 
For the squared Laplacian term, an evolution equation for the boundary normal derivative is proposed consistent with the energy dissipation law. 
The numerical scheme is designed with special attention to the coupled constraints by discretizing the quasi-entropy implicitly.
Spatial discretizations are carefully written down to be consistent with the boundary evolution equation, keeping the dissipation structure. 
As a result, we establish the existence and uniqueness, energy dissipation and error estimates for both the time discretization and the full discretization. 
The scheme proves to be crucial for the efficiency and robustness in the computation, especially for maintaining the coupled constraints that would be otherwise difficult to attain. 
Defects are found distinct from other layer structures due to the presence of local anisotropy. 

The entropy term proposed in the current work can be extended to handle liquid crystals formed by rigid molecules allowing $SO(3)$ rotations, which may possess complex architectures \cite{Xu2020,Xu_2020,xu_2022_kernel,Xu2022Quasi}. 
Another significant problem is to incorporate dissipation operators derived from the microscopic theory, which typically involve high order tensors.

\textbf{Acknowledgements.}  This work is partially supported by Beijing Natural Science Foundation (No. JQ25002), National Natural Science Foundation of China (Nos. 12301552, 12288201, 12371414), the Strategic Priority Research Program of the Chinese Academy of Sciences (No. XDB0510201), National Key R\&D Program of China (No. 2023YFA1008802).

\appendix
\section{Summation by parts}
Firstly, we derive the equality in (\ref{Dkc}). By direct computation,
\begin{equation}\label{m}
\sum_{l=0}^{N-1}(u_{l+1,m}-u_{l,m})(v_{l+1,m}-v_{l,m})                                                       =\sum_{l=1}^{N}-h^2u_{l,m}D_{11}v_{l,m}-u_{0,m}(v_{1,m}-v_{0,m})+u_{N,m}(v_{N,m}-v_{N-1,m}).
\end{equation}
Substituting (\ref{D11_0m}) and corresponding identities into (\ref{m}), and then substituting the resulting expression into (\ref{D_1D_1}), we have
\begin{equation}\label{D1uD1v}
       \sum_\Gamma{(D_1uD_1v)}_\Gamma
    =  -\sum_\Gamma{(uD_{11}v)}_\Gamma+{1\over 2h}\sum_{m=0}^{N-1}\Big({(uD_\bfn^x v)}_{0,m}+{(uD_\bfn^x v)}_{0,m+1}+{(uD_\bfn^x v)}_{N,m}+{(uD_\bfn^x v)}_{N,m+1}\Big).
\end{equation}
Similarly, we have 
\begin{equation}\label{D2uD2v}
\sum_\Gamma{(D_2uD_2v)}_\Gamma                                                                              =  -\sum_\Gamma{(uD_{22}v)}_\Gamma+{1\over 2h}\sum_{l=0}^{N-1}\Big({(uD_\bfn^y v)}_{l,0}+{(uD_\bfn^y v)}_{l+1,0}+{(uD_\bfn^y v)}_{l,N}+{(uD_\bfn^y v)}_{l+1,N}\Big).
\end{equation}
The sum of (\ref{D1uD1v}) and (\ref{D2uD2v}) gives (\ref{Dkc}).

Next, we focus on the equality in (\ref{Dlc}). For $m=0,\ldots,N$,
combining with the boundary conditions $v_{0,m}=v_{N,m}=0$, we derive
\begin{equation}\label{m21}
  \begin{aligned}
      & \sum_{l,m=0}^{N-1}{\big({D_{11}uD_{11}v}\big)}_{l,m}\\
    = & \sum_{l,m=0}^{N-1}{\big(D_{11}u\big)}_{l,m}{v_{l-1,m}-2v_{l,m}+v_{l+1,m}\over h^2} \\
    = & \sum_{m=0}^{N-1}\Big(\sum_{l=1}^{N-1}{{\big(D_{11}u\big)}_{l+1,m}-2{\big(D_{11}u\big)}_{l,m}+{\big(D_{11} u\big)}_{l-1,m}\over h^2}v_{l,m} +{\big({D_{11}u}\big)}_{0,m}{v_{-1,m}\over h^2}-{\big({D_{11}u}\big)}_{N,m}{v_{N-1,m}\over h^2}\Big)\\
    = & \sum_{l,m=1}^{N-1}{\big(D_{1111}u\big)}_{l,m}v_{l,m}+\sum_{m=0}^{N-1}\Big({\big({D_{11}u}\big)}_{0,m}{v_{-1,m}\over h^2}-{\big({D_{11}u}\big)}_{N,m}{v_{N-1,m}\over h^2}\Big).
  \end{aligned}
\end{equation}
Similarly, together with the boundary conditions $v_{l,0}=v_{l,N}=0$ for $l=0,\ldots,N$, we have
\begin{equation}\label{m22}
\begin{aligned}
       &\sum_{l,m=0}^{N-1}{\big({D_{11}uD_{11}v}\big)}_{l+1,m}  =  \sum_{l,m=1}^{N-1}{\big(D_{1111}u\big)}_{l,m}v_{l,m}+\sum_{m=0}^{N-1}\Big(-{\big({D_{11}u}\big)}_{0,m}{v_{1,m}\over h^2}+{\big({D_{11}u}\big)}_{N,m}{v_{N+1,m}\over h^2}\Big), \\  
       &\sum_{l,m=0}^{N-1}{\big({D_{11}uD_{11}v}\big)}_{l,m+1}  =  \sum_{l,m=1}^{N-1}{\big(D_{1111}u\big)}_{l,m}v_{l,m}+\sum_{m=0}^{N-1}\Big({\big({D_{11}u}\big)}_{0,m+1}{v_{-1,m+1}\over h^2}-{\big({D_{11}u}\big)}_{N,m+1}{v_{N-1,m+1}\over h^2}\Big),\\
       &\sum_{l,m=0}^{N-1}{\big({D_{11}uD_{11}v}\big)}_{l+1,m+1}  =  \sum_{l,m=1}^{N-1}{\big(D_{1111}u\big)}_{l,m}v_{l,m}+\sum_{m=0}^{N-1}\Big(-{\big({D_{11}u}\big)}_{0,m+1}{v_{1,m+1}\over h^2}+{\big({D_{11}u}\big)}_{N,m+1}{v_{N+1,m+1}\over h^2}\Big).
\end{aligned}
\end{equation}
Substituting (\ref{m21}) and (\ref{m22}) into the summation of (\ref{D11uD11v}) in space, we derive that
\begin{equation}\label{DD1}
  \begin{aligned}
      & \sum_\Gamma{(D_{11}uD_{11}v)}_\Gamma                                                                                                                                                                   \\
    = & \sum_{l,m=1}^{N-1}{\big(vD_{1111}u\big)}_{l,m}+\sum_{m=0}^{N-1}{1\over 2h}\Big({\big({D_{11}u}D_\bfn^x v\big)}_{0,m}+{\big({D_{11}u}D_\bfn^x v\big)}_{0,m+1}                 
       +{\big({D_{11}u}D_\bfn^x v\big)}_{N,m}+{\big({D_{11}u}D_\bfn^x v\big)}_{N,m+1}\Big).
  \end{aligned}
\end{equation}
In similar ways, we can also derive that
\begin{equation}\label{DD2}
  \begin{aligned}
    \sum_\Gamma{(D_{22}uD_{11}v)}_\Gamma= 
     & \sum_{m=1}^{N-1}\sum_{l=1}^{N-1}{\big(vD_{2211}u\big)}_{l,m}+\sum_{m=0}^{N-1}{1\over 2h}\Big({\big({D_{22}u}D_\bfn^x v\big)}_{0,m}+{\big({D_{22}u}D_\bfn^x v\big)}_{0,m+1} \\
     & +{\big({D_{11}u}D_\bfn^x v\big)}_{N,m}+{\big({D_{11}u}D_\bfn^x v\big)}_{N,m+1}\Big),                                                                                              \\
    \sum_\Gamma{(D_{11}u D_{22}v)}_{\Gamma}= 
     & \sum_{m=1}^{N-1}\sum_{l=1}^{N-1}{\big(vD_{1122}u\big)}_{l,m}+\sum_{l=0}^{N-1}{1\over 2h}\Big({\big({D_{11}u}D_\bfn^y v\big)}_{l,0}+{\big({D_{11}u}D_\bfn^y v\big)}_{l+1,0} \\
     & +{\big({D_{11}u}D_\bfn^y v\big)}_{l,N}+{\big({D_{11}u}D_\bfn^y v\big)}_{l+1,N}\Big),                                                                                              \\\sum_\Gamma{(D_{22}uD_{22}v)}_\Gamma=&\sum_{m=1}^{N-1}\sum_{l=1}^{N-1}{\big(vD_{2222}u\big)}_{l,m}+\sum_{l=0}^{N-1}{1\over 2h}\Big({\big({D_{22}u}D_\bfn^y v\big)}_{l,0}+{\big({D_{22}u}D_\bfn^y v\big)}_{l+1,0}\\
     & +{\big({D_{22}u}D_\bfn^y v\big)}_{l,N}+{\big({D_{22}u}D_\bfn^y v\big)}_{l+1,N}\Big).                                                                                              \\
  \end{aligned}
\end{equation}
Thus, the sum of (\ref{DD1}) and (\ref{DD2}) proves (\ref{Dlc}).

\bibliographystyle{plain}
\bibliography{convex-analysis}
\end{document}